\documentclass[reqno,11pt]{amsart}
\usepackage[utf8x]{inputenc}
\usepackage{amsmath,amssymb,latexsym,esint,cite,mathrsfs, tikz}
\usepackage{verbatim,wasysym}
\usepackage[left=3cm,right=3cm,top=3cm,bottom=3cm]{geometry}
\usetikzlibrary{decorations.pathreplacing, calligraphy}

\usepackage{microtype}
\usepackage{color,enumitem,graphicx}
\usepackage[colorlinks=true,urlcolor=blue, citecolor=red,linkcolor=blue,
linktocpage,pdfpagelabels, bookmarksnumbered,bookmarksopen]{hyperref}
\usepackage[hyperpageref]{backref}
\usepackage[english]{babel}

\usepackage[colorinlistoftodos]{todonotes}

\newtheorem{lemma}{Lemma}[section]
\newtheorem{corollary}[lemma]{Corollary} 
\newtheorem{proposition}[lemma]{Proposition}
\newtheorem{theorem}[lemma]{Theorem}
\newtheorem{ex}[lemma]{Example}

\theoremstyle{definition}

\newtheorem{remark}[lemma]{Remark}
\newtheorem{defn}[lemma]{Definition}

\newcommand{\R}{\mathbb{R}}
\newcommand{\N}{\mathbb{N}}

\newcommand{\eps}{\varepsilon}
\newcommand{\beq}{\begin{equation}}
\newcommand{\eeq}{\end{equation}}

\newcommand{\RN}{\mathbb{R}^N}

\def\XXint#1#2#3{{\setbox0=\hbox{$#1{#2#3}{\int}$ }
\vcenter{\hbox{$#2#3$ }}\kern-.6\wd0}}

\numberwithin{equation}{section}

\title[A non-smooth Brezis-Oswald uniqueness result]{A non-smooth Brezis-Oswald uniqueness result}

\author[S. \ Mosconi]{Sunra Mosconi}

\address[S.\ Mosconi]{Department of Mathematics and Computer Science
	\newline\indent
	University of Catania
	\newline\indent
	Viale A. Doria 6, I-95125 Catania, Italy}
\email{sunra.mosconi@unict.it}

\date{\today}

\begin{document}

\begin{abstract}
We classify the non-negative critical points in $W^{1,p}_0(\Omega)$ of 
\[
J(v)=\int_\Omega H(Dv)-F(x, v)\, dx
\]
where $H$ is convex and  positively $p$-homogeneous, while $t\mapsto \partial_tF(x, t)/t^{p-1}$ is non-increasing. Since $H$ may not be differentiable and $F$ has a one-sided growth condition, $J$ is only l.\,s.\,c.\,on $W^{1,p}_0(\Omega)$. We employ a weak notion of critical point for non-smooth functionals, derive sufficient regularity of the latter without an Euler-Lagrange equation available and focus on the uniqueness part of the results in \cite{BO}, through a non-smooth Picone inequality.
  \end{abstract}

\maketitle

	\begin{center}
		\begin{minipage}{9cm}
			\small
			\tableofcontents
		\end{minipage}
	\end{center}
 
 \section{Introduction}

The celebrated Brezis-Oswald theorem \cite{BO} states, among other things, that the problem
\beq
\label{bo}
\begin{cases}
-\Delta u=f(x, u) &\text{in $\Omega$}\\
u=0&\text{on $\partial\Omega$}\\
u\ge 0 &\text{in $\Omega$}
\end{cases}
\eeq
in a connected smooth bounded domain $\Omega\subseteq \R^N$ has a unique solution whenever $f(\cdot, t)\in L^\infty(\Omega)$ for all $t\ge 0$ and $t\mapsto f(x, t)/t$ is strictly decreasing in $t>0$, for a.\,e.\,$x\in \Omega$. This result has later been generalised to the quasi-linear setting in \cite{DS}, to the nonlocal one (see \cite{IM}) and even to mixed local-nonlocal problems in \cite{BMV}, producing countless applications in nonlinear analysis.

The general principle underlying Brezis-Oswald theorem appears to be the following. If $L$ denotes a (pseudo)-differential elliptic operator enjoying a strong minimum principle and which is {\em homogeneous} of degree $q$, then uniqueness of the solution 
\beq
\label{Lq}
\begin{cases}
L(u)=f(x, u) &\text{in $\Omega$}\\
u=0&\text{on $\partial\Omega$}\\
u\ge 0 &\text{in $\Omega$}
\end{cases}
\eeq
 should follow under the assumption that $t\mapsto f(x, t)/t^{q}$ is strictly decreasing. These kind of results, therefore, are called in the literature "Brezis-Oswald uniqueness theorems".

A number of additional assumptions, however, are generally required for such statements to hold, and we discuss here a few of them.
\begin{enumerate}
\item
{\em Regularity of the domain}.
Often, the domain $\Omega$ is assumed to have smooth (say, $C^{1, \alpha}$) boundary, so that a Hopf-type maximum principle holds  true for the corresponding solutions. This usually requires rather fine regularity theory for solutions of the corresponding problems, and is usually avoided requiring additional technical assumptions on the reaction (see e.\,g.\,\cite{BMV}). 
\item
{\em Growth conditions on $f$}.
The monotonicity assumption on $t\mapsto f(x, t)/t^q$ only grants an upper bound on the growth of $f$. Even if the operator $L$ is variational, namely $L= dE$ for a suitable differentiable convex functional $E:X\to \R$ on a Banach space $X\hookrightarrow L^p(\Omega)$, this poses delicate issues. Indeed, it is then natural to consider solutions of \eqref{Lq} as critical points of the functional
\[
J(v)=E(v)-\int_\Omega F(x, u)\, dx, \qquad F(x, t)=\int_0^t f(x, s)\, ds
\]
but without assuming any further lower growth control  on $f$, $J$ is, typically, only lower semicontinuous.  This problem is often side-stepped by either requiring a-priori boundedness in the notion of solution of \eqref{Lq} (so that actually $f$ can be truncated), or by directly requiring a two sided growth condition on $f$.
\item
{\em Monotonicity assumption}. It is clear that uniqueness fails even for \eqref{bo} without {\em strict} monotonicity, as seen by considering the case $f(x, t)=\lambda_1(\Omega)\, t$, where $\lambda_1(\Omega)$ is the first eigenfunction of the Laplacian with Dirichlet boundary conditions. In this case $t\mapsto f(x, t)/t$ is only non-increasing and the corresponding problem \eqref{bo} has infinitely many positive solutions, namely all first eigenfunctions of the Dirichlet Laplacian. However, at least for autonomous reactions independent of $x$,  this is essentially the only obstruction to uniqueness. Indeed, in \cite[Proposition 3.8]{BMS1} it is proved that, for $L=-\Delta_{p} $ and $t\mapsto f(t)/t^{p-1}$ only {\em non-increasing} in $t>0$, either there is a unique solution to the corresponding problem \eqref{Lq}, or all solutions are first eigenfunctions for the Dirichlet $p$-Laplacian in $\Omega$. In particular, uniqueness holds true if $t\mapsto f(t)/t^{p-1}$ is non-increasing and strictly decreasing only on an arbitrarily small interval $]0, \delta[$, $\delta>0$.
\item
{\em Regularity of $L$}. As anticipated in point (2) above, the typical framework for a variational Brezis-Oswald uniqueness result is $L(u)=dE(u)$, for $E(u)=\|u\|_X^p$  and  $\|\ \|_X$ is the norm of a Banach space  $X\hookrightarrow L^p(\Omega)$, smooth away from the origin. Differentiability of the norm is typically used to have a single valued operator $L$ in equation \eqref{Lq}. Considering non-differentiable and only locally Lipschitz energies $E$ yields a differential inclusion formally of the form
\[
f(x, u)\in \partial  E(u)
\]
for critical points of the relevant functional, where here $\partial$ denotes the multivalued convex subdifferential. In this more general setting,  uniqueness results are clearly stronger statements.
\end{enumerate}

\subsection{Main results}

In this manuscript, we plan to address the previous themes,  aiming to obtain a general version of the uniqueness part of Brezis-Oswald theorem. In particular, we seek for a classification of the critical points of functionals of the form
\beq
\label{J}
J(v)=\int_\Omega H(Dv) - F(x, v)\, dx, \qquad v\in W^{1,p}_0(\Omega)
\eeq
under fairly general conditions on $\Omega$, $H$,  and $f(x, t)=\partial_t F(x, t)$. Our main result reads as follows (for a more complete statement, see Theorem \ref{thfin}).

\begin{theorem}\label{Mth}
Let $\Omega\subseteq \R^N$ be open, connected and of finite measure.
For some $p>1$ suppose that 
\begin{itemize}
\item[i)] $H:\R^N\to [0, \infty[$ is strictly convex and positively $p$-homogeneous
\item[ii)] $f:\Omega\times \R\to \R$ is Caratheodory and fulfills 
\[
\sup_{x\in \Omega, |s|\le t} |f(x, s)|<\infty\qquad \forall t>0
\]
\item[iii)]
 The map $ t\mapsto f(x, t)/t^{p-1}$ is non-increasing on $]0, \infty[$.
\end{itemize}
For $J$ as in \eqref{J} with 
\[
F(x, t)=\int_0^t f(x, s)\, ds,
\]
let  $C_+(J)$ be the set of  its non-negative critical points  $v\in W^{1,p}_0(\Omega)\setminus\{0\}$. Then
\begin{enumerate}
 \item
 All elements in $C_+(J)$ are multiple of each other and each of them minimise $J$ among non-negative functions in  $ W^{1,p}_{0}(\Omega)$ 
 \item
 If $t\mapsto f(x, t)/t^{p-1}$ is strictly decreasing on $t>0$,  $C_+(J)$ contains at most one element
 \item
 If $f$ does not depend on $x$ and $C_+(J)$ contains more than one element, each function in $C_+(J)$ minimises the Rayleigh quotient
 \beq
 \label{l0}
\frac{\displaystyle{\int_\Omega H(Dv)\, dx}}{\displaystyle{\int_\Omega v^p\, dx}}  
\eeq
among  all $v\in W^{1,p}_0(\Omega)\setminus \{0\}$,   $v\ge 0$.
 \end{enumerate}
\end{theorem}

  \begin{remark}
  Let us make some comments on the features of the previous theorem.
  \begin{itemize}
  \item
  No regularity assumptions are made on the boundary of the open set $\Omega$. The finite measure assumption could be further weakened requiring additional conditions on the behaviour of $f(x, t)$ near $0$, but we chose not to analyse this case in order to have a cleaner statement. Notice that in this setting, there is no hope to have a Hopf maximum principle available.
   \item
   The stated assumption on $H$ enforces the two sided bound
   \[
   c\, |z|^p\le H(z)\le d\, |z|^p, \qquad \forall z\in \R^N
   \]
   for some $0<c\le d<\infty$. However, since $H$ is not supposed to be $C^1$, the first term of $J$ in \eqref{J} is only locally Lipschitz, with multi-valued differential. Notice also that $H$ is not supposed to be even. The strict convexity of $H$ is also not needed in assertion {\em (2)}, as well as in the second statement of assertion {\em (1)}.
   \item
   Conditions {\em ii)} and {\em iii)} on the reaction $f$ are equivalent to 
   \[
   f(\cdot, t)\in L^\infty(\Omega), \qquad f(x, t)\le C\, (1+t^{p-1}), \qquad t\mapsto \frac{f(x, t)}{t^{p-1}} \ \text{non-increasing}
   \]
   for $t> 0$, which are more commonly found in Brezis-Oswald type results.
   \item
  Under the sole assumptions {\em ii)} and {\em iii)}, even for smooth $H$, $J$ is only l.\,s.\,c.,  and the notion of critical point has to be dealt with care. We employ here the notion of {\em energy critical point} used in \cite{DGM}, meaning those $u\in W^{1,p}_0(\Omega)$ which, roughly speaking, minimise the corresponding semi-linearised convex functional
\[
v\mapsto \int_\Omega H(Dv) - f(x, u)\, v\, dx.
\]
For the sake of this introductory discussion, let us note that the usually encountered notions of critical point, when available, yield an energy critical point.
For more details, we refer to Definition \ref{defecp} and Remark \ref{remc}.
\item
Notice that we don't give any condition ensuring that $C_+(J)$ is actually not empty. One of the main features of the original Brezis-Oswald result (and of the related literature on the subject) was actually to provide necessary and sufficient condition for $J$ to have non-negative, non-trivial critical points, under a strict monotonicity assumption on $t\mapsto f(x, t)/t^{p-1}$. We won't investigate this theme here.
\item
Under assumption {\em iii)} alone, the structure of $C_+(J)$ is described in an optimal way by conclusion {\em (1)} of the theorem. Indeed, when $H(z)=|z|^2$, given a first positive eigenfunction $u$ of the Dirichlet Laplacian in the unit ball $B_1$ and an interval $[a, b]\subseteq [0, +\infty[$, any reaction $f$ obeying $f(x, t)=\lambda_1(B_1)\, t$ if $a\, u(x)\le t\le b\, u(x)$ and $f(x, t)/t$ strictly decreasing otherwise has $C_+(J)=\{t\, u: t\in [a, b]\}$.
\item
A similar statement holds true for non-positive critical points of $J$, with obvious modifications to {\em ii)} and {\em iii)}. Notice, however, that since we are not assuming that $H$ is even, the infimum of the ratios \eqref{l0} may change under the corresponding constrain $v\le 0$. This feature of possibly non-even, positively $p$-homogeneous energies $H$ and their corresponding first Dirichlet eigenvalues is discussed in section \ref{DE}.

\end{itemize}
\end{remark}

In order to prove Theorem \ref{Mth}, we need to develop basic regularity theory for energy critical points. The main step is contained in the following {\em a-priori} bound, which is of independent interest. In its statement, we let as usual $p^*=N\, p/(N-p)$ if $p<N$, $p^*=\infty$ if $p>N$ and $p^*$ be any finite number if $p=N$.

\begin{theorem}\label{Mth2}
Let $\Omega\subseteq \R^N$ have finite measure and, for given $p\ge 1$, suppose that 
\begin{itemize}
\item[i)] $H:\R^N\to [0, \infty[$ is convex and fulfills $H(z)\ge c\, |z|^p$ for some $c>0$ and all $z\in \R^N$
\item[ii)] $f:\Omega\times \R\to \R$ is Caratheodory and fulfills 
\[
\sup_{x\in \Omega, |s|\le t} |f(x, s)|<\infty\qquad \forall t>0
\]
\item[iii)]
There exist $\mu, \theta\ge 0$ and $1\le q\le p^*$ such that
\beq
\label{fgh}
{\rm sign}\,(t) f(x, t)\le \mu\, (|t|+\theta)^q\qquad \text{for a.\,e.\,$x\in \Omega$ and all $t\in \R$}.
\eeq
\end{itemize}
Then, any critical point $u\in W^{1,p}_0(\Omega)$ of $J$
is bounded. If additionally $q<p^*$, the estimate
\beq
\label{supboundi}
\|u\|_{L^\infty(\Omega)}\le 
 \max\left\{\theta, C\,\left((\mu/c)^{\frac{N}{p}}\int_\Omega |u|^q\, dx\right)^{\frac{p}{p\, q+ (p-q)\, N}}\right\} 
 \eeq
holds true for some $C=C(N, p, q)$ which, for any $N$, is positive and continuous in its other arguments.
\end{theorem}

\begin{remark}
Regarding this second result, the following should be noted.
\begin{itemize}
\item
Not only $H$ is only assumed to be locally Lipschitz, but no growth control from above is required in order to obtain boundedness of the critical points of $J$. 
\item
Assumption {\em iii)} is understood to be empty in the case $q=\infty$, which forces $p>N$ thanks to $q\le p^*$. Clearly, in this case boundedness is trivial thanks to Morrey embedding. Anyway, \eqref{supboundi} only holds for $q<p^*$ and therefore for finite $q$'s.
\item
Assumption {\em ii)} and {\em iii)} are required in order for the functional $J$ in \eqref{J} to be well defined (possibly assuming the value $+\infty$) on $W^{1,p}_0(\Omega)$. In any case, under these assumptions $J$ turns out to be only l.\,s.\,c.\,on $W^{1,p}_0(\Omega)$ with the strong topology, and again we understand by critical point an energy one (see subsection \ref{nscp}).
\item
The form  \eqref{supboundi} of the sup bound  is optimal, as discussed in Remark \ref{remsup}. Moreover, it is invariant by scaling and scalar multiples, meaning the following. A function  $u$ is a critical point for $J$ on $W^{1,p}_0(\Omega)$ if and only if for any $\lambda, r>0$ $u_{\lambda, r}(x)=\lambda\, u(x/r)$ is a critical point for a suitable functional $J_{\lambda, r}$ on $W^{1,p}_0(r\, \Omega)$ (where $r\, \Omega=\{r\, x: x\in \Omega\}$). More precisely, the functional $J_{\lambda, r}$ is of the form \eqref{J} for a reaction $\hat{f}$ obeying \eqref{fgh} with $\hat{\theta}=\lambda\, \theta$, $\hat{\mu}=\mu/\lambda^q$ and and $\hat{H}$ obeying $\hat{H}(z)\ge \hat{c}\, |z|^p$ with $\hat{c}=c\, \lambda/r$. The corresponding estimates \eqref{supboundi} turn out to be independent of $\lambda$ and $r$.
\item
Regarding the continuity statement for the constant $C$ in the last part of the theorem, it is to be understood in its domain $p\ge 1$, $q<p^*$. As is to be expected, the constant $C$ blows up as  $q\uparrow p^*$.
\end{itemize}
\end{remark}

\subsection{Related results and motivations}

Applications and extensions of the Brezis-Oswald theorem are countless and it would be impossible to list all the contributions on the subject.
The principal part of the operator  considered in the various generalisation can be classical quasi-linear operators, as in \cite{DS}, nonlocal operators as in \cite{IM} and even mixed local-nonlocal ones as in \cite{BMV}. In all these cases the operator is $(p-1)$-homogeneous and the corresponding monotonicity assumption on the reaction pertains $t\mapsto f(x, t)/t^{p-1}$. Non-homogeneous operators of Orlicz type are considered in \cite{CGSS}. The case when $t\mapsto f(x, t)/t^{p-1}$ is non-increasing but not strictly decreasing is naturally related to weighted nonlinear eigenvalue problems, where $f(x, t)=a(x)\, t^{p-1}$. Corresponding eigenfunctions are considered in \cite{BF,T, TTU} when the energy density $H$ is $p$-homogeneous and $C^1$, while the case where $H$ may be non-differentiable is studied in \cite{LS}. Various applications to nonlinear parabolic problems, encompassing the non-differentiable setting, are described in \cite{D}.

A complete Brezis-Oswald theorem usually consists of two statements. The first one is devoted to derive necessary and sufficient condition for the solvability of \eqref{Lq}, under the aforementioned monotonicity assumption  of $t\mapsto f(x, t)/t^{p-1}$, as is done in \cite{BO, DS, IM}. The second one regards uniqueness of the solution, and is clearly the focus of this research.

As briefly outlined in the first section, our main intent with this investigation was to clarify how the various technical assumptions which usually appear in Brezis-Oswald type result could  be weakened to have a cleaner picture, especially regarding its uniqueness part.  It is clear, for instance, that the regularity of the domain or of $H$ is not necessary for the uniqueness statement to hold. Our interest in this part of Brezis-Oswald  theorem, however, lies mainly in its applications to the qualitative study of solutions to PDE's.

Indeed, uniqueness is the basic step needed to effectively approximate the particular PDE problem under scrutiny of the form \eqref{Lq} with solutions of regularised problems
\beq
\label{Lqn}
L_n(u)=f_n(x, u),
\eeq
each having a smooth solution $u_n$ allowing computations can be more effectively performed.  The sequence $(u_n)$  then shown to enjoy uniform bounds and therefore to converge up to subsequences to a solution of the limiting, original problem \eqref{Lq}. Then, uniqueness of solutions of \eqref{Lq} is the key point ensuring that the approximating solutions actually converge to the original one, which then (hopefully) inherits some properties of the $u_n$'s. This scheme of proof is nowadays standard in both regularity theory and qualitative properties of solution to PDE's. 

As a typical example, let us consider positive solutions of the quasi-linear problem
\[
-\Delta_pu =f(u)
\]
with Dirichlet boundary conditions and a qualitative property of these solutions known as {\em quasi-concavity}, meaning that the super-level sets of $u$ are convex. 
It is quite fortunate that the monotonicity assumption on $t\mapsto f(t)/t^{p-1}$ is strictly connected to the general conditions found in \cite{BMS1} ensuring quasi-concavity of the aforementioned solutions. Actually, the results in \cite{BMS1, BMS2} are proved via the smoothing technique outlined above under the assumption that the domain is convex with $C^{1,\alpha}$ boundary, but Theorem \ref{Mth} allow to extend the results to  general convex domains.

As a motivation for Theorem \ref{Mth2}, let us also stress that, in order for the previously described approximation scheme to work, uniform bounds are also essential to ensure compactness. This is especially true when, in the family of  problems \eqref{Lqn}, both the operators $L_n$ and the reactions $f_n$ may have a varying  growth control, only qualitatively bounded. In this regard, Theorem \ref{Mth2} is inspired by \cite{BL}, where the stability of the available a-priori bounds had to be treated with extreme care for the employed contradiction argument to carry on. In particular, \eqref{supboundi} should be compared with  \cite[Proposition 2.4]{BL}, corresponding to the case $\theta=0$ in \eqref{fgh}. In \cite{BL}, stability with respect to $p$ wasn't an issue, and the a-priori bound was derived via a Moser iteration technique involving the $p$-Sobolev inequality. Since the constant in the  latter blows up as $p\uparrow N$, the obtained bounds in \cite{BL} are not stable in this limit. To be honest, we initially hoped for this apparent instability to be substantial, suggesting the possibility of interesting phenomena exhibited by solutions of, say, the $p$-Laplace equations as $p\uparrow N$. Unfortunately, a different and, we think, more streamlined  proof of the $L^\infty$ bound in the spirit of De Giorgi method, shows that this is not the case.

Let us finally mention that our proof of Theorem \ref{Mth}, detailed in the next section, relies on a non-smooth version of the Picone inequality which relies only on the positive $p$-homogeneity of $H$. This should be compared to  \cite{J} where, however, differentiable convex $H$ are considered. A natural alternative route to the Brezis-Oswald uniqueness result, as in \cite{D, DS, TTU}, is to consider the hidden convexity properties of the functional $J$ in \eqref{J}. Indeed, the core assumption that $t\mapsto f(x, t)/t^{p-1}$ be non-increasing is implied by the weaker hypothesis that   $t\mapsto F(x, t^{1/p})$ is concave. Notice that this latter condition does not require $F$ to be differentiable in its second variable, nor any  bound on its derivative is  needed to define the functional $J$. Concavity of $t\mapsto F(x, t^{1/p})$ implies {\em hidden convexity} of 
\[
 J_F(v)=-\int_\Omega F(x, v)\, dx,
\]
on the cone of non-negative $v\in W^{1,p}_0(\Omega)$, which means convexity of $J_F$ along the curves
\[
t\mapsto \left(t \, u^{p}+(1-t)\, v^{p}\right)^{1/p} \qquad \forall u, v\in W^{1,p}_0(\Omega), \quad u, v\ge 0.
\]
On the other hand, by \cite{BF}, our non-smooth Picone inequality is equivalent to the hidden convexity of  $J_H$. Hence the whole $J$ turns out to be hidden convex on the cone of nonnegative functions in $W^{1,p}_0(\Omega)$ and it is only natural to expect few critical points from such a functional. We chose not to adopt this approach since our notion of energy critical point, while  very weak, requires at least that the reaction $f(x, t)=\partial_t F(x, t)$ is well defined and Caratheodory,  but  we point out that   the interplay between hidden convexity and the notion of critical point for l.\,s.\,c.\,functionals in the sense of \cite{DG} has been initiated in \cite{DGM2} in the framework of nonlinear eigenvalue problems.

\subsection{Outline of the proof}
The proof of Theorem \ref{Mth} is split in two main parts: regularity and uniqueness. The first obstacle to regularity is the lack of a proper Euler-Lagrange equation to work with, in particular when dealing with critical points of the l.\,s.\,c.\,functional $J$. Here we take advantage of the variational characterisation of critical points $u$ as minima of the semi-linearised functional
\[
J_u(v)=\int_\Omega H(Dv)-f(x, u)\, v\, dx.
\]
This allows to test the minimality of $J_u(u)$  against $J_u(u_k)$, where $u_k:=\max\{k, u\}$ for $k\ge 0$, and $J_u(u)\le J_u(u_k)$ reads
\[
\int_\Omega H(D(u-k)_+)\, dx\le \int_\Omega f(x, u)\, (u-k)_+\, dx.
\] 
This inequality and the one-sided growth control on $f$ suffice to derive a "fast decay" relation of the $L^q$ norms of $(u-k)_+$ much in the spirit of De Giorgi regularity proof. However, in order to ensure stability of the estimates in $p$, we avoid the natural Sobolev inequality and instead use an interpolation inequality {\em \'a la} Gagliardo-Nirenberg with a more tamed constant (see Proposition \ref{PGN}). Once the $L^\infty$ bound of Theorem \ref{Mth2} is proved, the reaction is bounded and Giaquinta-Giusti hole-filling method \cite{GG} allows to prove that the solutions  belong to suitable De Giorgi classes, thus providing continuity and a strong minimum principle for non-negative critical points. In particular, an Euler-Lagrange equation (or rather, inclusion) is satisfied by critical points, namely
\beq
\label{eli}
\int_\Omega (\xi, D\varphi)=\int_\Omega f(x, u)\, \varphi\, dx\qquad \forall \varphi\in W^{1,p}_0(\Omega)
\eeq
where $\xi:\Omega\to \R^N$ is a measurable choice of the multifunction $x\mapsto \partial H(Du(x))$ and $(\ , \ )$ denotes the Euclidean scalar product.

With these regularity tools at our disposal we can prove the uniqueness part. First we derive a non-smooth version of Picone inequality in Lemma \ref{thpicone}. If $H$  positively $p$-homogeneous  and convex, the point-wise inequality
\beq
\label{as}
p\, H(z)\ge (\xi, z), \qquad \forall \xi\in \partial H(z)
\eeq
holds true.
The non-smooth Picone inequality follows from the latter and reads as follows:  if $u$ and $v$ are positive functions in $W^{1,1}_{\rm loc}(\Omega)$, then a.\,e.\,in $\Omega$ it holds
\[
p\, H(Dv)\ge \left( \xi, D \frac{v^p}{u^{p-1}}\right), 
\]
for any measurable $\xi:\Omega\to \R^N$ such that $\xi(x) \in \partial H(Du(x))$ a.\,e.. Now, if $H$ is differentiable its convex subdifferential is single-valued, i.\,e.\,$\partial H=\{DH\}$ and  actually  \eqref{as} is an equality, so the smooth Picone inequality becomes
\[
\left( DH(Dv), Dv\right)\ge \left(DH(Du), D\frac{v^p}{u^{p-1}}\right).
\]
If both $u$ and $v$ are positive critical points, \eqref{eli} gives (cheating a little bit, since we are testing it for $u$ with $\varphi=v^p/u^{p-1}$)
\[
\int_\Omega f(x, v)\, v\, dx\ge \int_\Omega f(x, u)\, \frac{v^p}{u^{p-1}}\, dx.
\]
Exchanging the role of $u$ and $v$ and summing the corresponding inequalities, one arrives at
\beq
\label{pl}
\int_\Omega\left(\frac{f(x, u)}{u^{p-1}}-\frac{f(x, v)}{v^{p-1}}\right) (u-v)\, dx\ge 0
\eeq
which yields uniqueness when $t\mapsto f(x,t)/t^{p-1}$ is strictly decreasing. 

The point, in the non-differentiable case, is that in general the inequality in \eqref{as} can be strict. However, it turns out that the measurable selection $\xi$ of $\partial H(Du)$ arising from \eqref{eli} actually attains equality in \eqref{as} almost everywhere. This property is related to our variational notion of critical point,  and allows to carry on the proof as us usual.
Finally, the characterisation of the set of critical points described in Theorem \ref{Mth} under the non-increasing assumption on $t\mapsto f(x, t)/t^{p-1}$ is based on the analysis of the equality case in the non-smooth Picone inequality, derived from the identity
\[
\frac{f(x, u)}{u^{p-1}}=\frac{f(x, v)}{v^{p-1}}
\]
which must hold a.\,e.\,if $t\mapsto f(x, t)/t^{p-1}$ is non-increasing, thanks to \eqref{pl}.

\subsection{Structure of the paper}
Section \ref{Spre} contains preliminary material: we collect the notations used in the paper, provide some basic functional inequalities, derive basic properties of convex, positively $p$ homogeneous functions and describe the notion of energy critical points and its basic features. Section \ref{Sregularity} contains the proof of Theorem \ref{Mth2} and some of its consequences, such as continuity and  minimum principles for energy critical points. In Section \ref{DE} we define and study the first Dirichlet eigenvalues associated to possibly non-even convex energies, aiming at the proof of point {\em (3)} of Theorem \ref{Mth}. In the final section we prove the non-smooth Picone inequality and Theorem \ref{Mth}.

\section{Preliminaries}\label{Spre}
 \subsection{Notations}

 Throughout the paper we will make use of the following notations.
 
 \medskip
 Given a number $a\in \R$, we set $a_+=\max\{a, 0\}$ and $a_-=-\min\{0, a\}$ and similarly for real valued functions.
For every $v \in \RN$ we denote by $|v|$ its Euclidean norm  while $(v, w)$ stands for the scalar product, $v, w \in \RN$. The symbol $B$ denotes a generic ball with unspecified center and radius while  $B_r$ denotes a ball of radius $r$ centred at the origin. $S_1$ stands for the unit sphere $\partial B_1$. 
 \medskip

 For a Lebesgue measurable $E\subseteq \R^N$, $|E|$ denotes its measure and we set $\omega_N=|B_1|$. We will omit the domain of integration when it is the whole $\RN$, if this cause no confusion. For a measurable $E\subseteq \R^N$, ${\mathcal L}(E; \R^M)$ will denote the set of Lebesgue measurable functions from $E$ to $\R^M$. Given an open set $\Omega\subseteq \R^N$, $f$ will denote a Caratheodory function $f:\Omega\times \R\to \R$ and
\beq
\label{Ff}
F(x, t)=\int_0^tf(x, s)\, ds\qquad \forall t\in \R.
\eeq
Given  $u\in {\mathcal L}(\Omega; \R)$, we'll denote by $f(\cdot, u)$ the measurable function $x\mapsto f(x, u(x))$ and similarly for $F(\cdot, u)$. The Lebesgue spaces are denoted as usual by $L^p(E)$ for $p\ge 1$, with norm  $\|\cdot \|_{L^p(E)}$, while $\sup_E u$ and $\inf_E u$ stand for the essential supremum and infimum respectively. The Lebesgue spaces of vector valued functions $u:E\to \R^M$, will be denoted by  $\left(L^p(E)\right)^M$ with norms $\| |u| \|_{L^p(E)}$, where $|u|$ is the standard Euclidean norm of $\R^M$. Furthermore, we set for the sake of notational simplicity $\|f\|_p:= \|f\|_{L^p(\RN)}$. 

\medskip
Given   $p\in [1, \infty]$, $p'$ denotes its conjugate, namely $p'=p/(p-1)$ if $p\in \, ]1, \infty[$, $1'=\infty$ and $\infty'=1$. Given an integer $N\ge 2$ and $p\in [1, \infty]$, we denote by 
\[
p^*=
\begin{cases}
\dfrac{N\, p}{N-p}&\text{if $p<N$}\\
\text{any positive number}&\text{if $p=N$}\\
\infty&\text{if $p> N$}.
\end{cases}
 \]
 For an open $\Omega\subseteq \R^N$ with finite measure, the Sobolev space $W^{1,p}_0(\Omega)$ is the closure of $C^\infty_c(\Omega)$ w.\,r.\,t.\,the seminorm $\|Du\|_p$. In the finite measure case, $W^{1,p}_0(\Omega)$ is a Banach space, continuously embedded in $L^q(\Omega)$ for all $q\in [1, p^*]$. Its dual will be denoted by $W^{-1, p'}(\Omega)$, with duality pairing  denoted by $\langle \Lambda, \varphi\rangle$ for $\Lambda\in W^{-1, p'}(\Omega)$ and $\varphi\in W^{1,p}_0(\Omega)$.  
  If $p<N$, the homogeneous Sobolev space $D^{1, p}(\R^N)$ is the space of functions in $L^1_{\rm loc}(\R^N)$ with distributional derivative in $L^p(\R^N)$ and which {\em vanish at infinity}, meaning that 
 \[
 |\{|u|>k\}|<\infty\qquad \text{ for all $k>0$}.
 \]
 
 \subsection{Functional inequalities}
 Functions in $D^{1,p}(\R^N)$ for $1\le p<N$ enjoy the Sobolev inequality 
\beq
\label{Sob}
 \|u\|_{p^*}\le C(N, p)\, \|Du\|_p.
\eeq
We are particularly interested in the $p=1$ case, for which $C(N, 1)=N^{-1}\, \omega_N^{-1/N}$ for all $N\ge 2$.

The following result is probably well known, but we couldn't find a precise references and thus provide its proof.

\begin{proposition}
For $1\le p, q<\infty$, let  $X_{p, q}(\R^N)$ be the vector space of functions in $L^q (\R^N)$ with distributional derivative in $L^p(\R^N)$, equipped with the norm $\|u\|=\|Du\|_p+\|u\|_q$. Then $C^\infty_c(\R^N)$ is dense in $X_{p, q}(\R^N)$.
\end{proposition}

\begin{proof}
The set $C^\infty(\R^N)\cap X_{p, q}$ is dense in $X_{p, q}$ by known results (see \cite[Theorem 1, Section 1.1.5]{Mazya}), so it suffices to show that any $u\in C^\infty(\R^N)\cap X_{p, q}$ can be approximated by functions in $C^\infty_c(\R^N)$. Fix a cutoff $\eta\in C^\infty_c(B_2; [0, 1])$ such that  $\eta\equiv 1$ in $B_1$   and, for any $u\in C^\infty(\R^N)\cap X_{p, q}$, define for $k\in \N$
\[
\eta_k(x)=\eta(x/k), \qquad u_k = u\, \eta_k,\qquad A_k=B_{2\, k}\setminus B_k,
\]
so that $u_k\in C^\infty_c(\R^N)$, ${\rm supp}\, (D\eta_k)\subseteq A_k$ and $\|D\eta_k\|_\infty=\|D\eta\|_\infty/k$. Clearly $\|u-u_k\|_q\le \|u\|_{L^q(\R^N\setminus B_k)}\to 0$ as $k\to \infty$, so we only have to estimate $\|D(u-u_k)\|_p$.   It holds
\beq
\label{split}
\begin{split}
\int |Du-Du_k|^p\, dx&\le C_p\int |Du|^p\, (1-\eta_k)^p\, dx+ C_p\, \int |u|^p\, |D\eta|^p\, dx\\
&\le C_p\, \int_{\R^N\setminus B_k} |Du|^p\, dx+ \frac{C_p\, \|D\eta\|_\infty^p }{k^p}\, \int_{A_k} |u|^p\, dx.
\end{split}
\eeq
and the first term on the right vanishes for $k\to \infty$ since $|Du|\in L^p(\R^N)$, so we are left to estimate the second term.

The case $p<N$ is standard: by Chebichev inequality any $u\in L^q(\R^N)$ vanishes at infinity, hence $X_{p, q}\subseteq D^{1,p}(\R^N)\subseteq L^{p^*}(\R^N)$.
Therefore, by H\"older inequality
\[
 \int_{A_k} |u|^p\, dx\le  \left(\int_{A_k}|u|^{p^*}\, dx\right)^{\frac{p}{p^*}}\, |A_k|^{1-\frac{p}{p^*}}=C\, k^p \, \left(\int_{A_k}|u|^{p^*}\, dx\right)^{\frac{p}{p^*}}, 
 \]
and the last factor on the right vanishes as $k\to \infty$ due to $u\in L^{p^*}(\R^N)$,  hence also the last term in \eqref{split} does so. If $p\ge N$, we use the inequality
\[
\int_{B_2\setminus B_1} |v|^p\, dx\le C\, \int_{B_2\setminus B_1} |Dv|^p\, dx +C\, \left(\int_{B_2\setminus B_1} |v|^q\, dx\right)^{p/q}
\]
which is trivial by H\"older when $q\ge p$ and holds true thanks to Gagliardo-Nirenberg inequality when $q<p$. Applying it to $v(x)=u(k\, x)$ and changing variables, we get
\[
 \frac{1 }{k^p}\, \int_{A_k} |u|^p\, dx\le C\, \int_{A_k} |Du|^p\, dx+ \frac{C\, k^{N\, \left(1-\frac{p}{q}\right)}}{k^p}\left(\int_{A_k} |u|^q\, dx\right)^{p/q}.
 \]
As $k\to \infty$, the first term on the right vanishes again, as well as the second, since   
\[
N\left(1-\frac{p}{q}\right)< p\qquad \forall p\ge N, q<\infty.
\]
\end{proof}

The next proposition is a particular case of the general Gagliardo-Nirenberg inequality. In order to get stable $L^\infty$ bounds in Section \ref{Sregularity}, however, we have to pay special attention to the dependance from the parameters and thus give a direct proof depending only on \eqref{Sob}.

\begin{proposition}[Gagliardo-Nirenberg]\label{PGN}
Let $N\ge 2$, $u\in W^{1,1}_{\rm loc}(\R^N)$ and $1\le p, q<\infty$. Then for any $u\in W^{1, 1}_{\rm loc}(\R^N)$ it holds
\beq
\label{GN}
\left(\int |u|^{\frac{N\, (q\, p+p-q)}{(N-1)\, p} }\, dx\right)^{\frac{N-1}{N}}\le G_{p, q, N}\left(\int |Du|^p\, dx\right)^{1/p}\left(\int |u|^q\, dx\right)^{1/p'}, 
\eeq
where
\[
G_{p, q, N}=\frac{1}{N\, \omega_N^{1/N}}\, \frac{q\, p+p-q}{p}.
\]
\end{proposition}

\begin{proof}
We can suppose that the right-hand side is finite, which amounts to $u\in X_{p, q}$. By the previous Proposition we can also suppose that $u\in C^\infty_c(\R^N)$.
For $p=1$ \eqref{GN} reduces to \eqref{Sob}, so we can suppose that $p>1$.
Let 
\[
r=\frac{q}{p'}+1=\frac{q\, p+p-q}{p}
\]
and apply \eqref{Sob} to $v=|u|^r$ and then H\"older inequality, to get
\[
\begin{split}
\left(\int  |u|^{\frac{q\, p+p-q}{p} 1^*}\, dx\right)^{1/1^*}&\le \frac{1}{N\, \omega_N^{1/N}}\, \int  |D|u|^r|\, dx=\frac{1}{N\, \omega_N^{1/N}}\, r\, \int  |u|^{r-1}\, |Du|\, dx\\
&\le \frac{1}{N\, \omega_N^{1/N}}\,  r\, \left(\int  |Du|^p\, dx\right)^{1/p}\left(\int  |u|^{q}\, dx\right)^{1/p'}.
\end{split}
\]
 \end{proof}
 
 \subsection{Convex $p$-homogeneous functions}
 
Let  $X$ be a Banach space in duality with $X^*$ via the duality pairing $\langle\cdot, \cdot\rangle$. Given a continuous convex function $C:X\to \R$ (we will alway restrict to this setting)   its {\em subdifferential} at $x\in X$ is
 \[
 \partial C(x)=\left\{ x^*\in X^*: C(x+v)- C(x)\ge \langle x^*, v\rangle \ \forall v\in X\right\}.
 \]
 When $X=\R^N$ we'll identify $X$ with $X^*$ via the scalar product as a duality pairing, hence in this case $\partial C(x)\subseteq \R^N$.
 
   Throughout the paper, $H:\R^N\to [0, \infty[$ will denote a convex function. Since its domain is $\R^N$, any such function is locally Lipschitz. The function $H$ is  positively $p$-homogeneous for some $p\ge 1$ if
   \[
   H(\lambda\, z)=\lambda^p\, H(z)\qquad \text{ for all $\lambda>0$ and $z\in \R^N$}.
   \]
   Notice that we do not assume regularity of $H$, nor that $H$ is even.
   In the positively $p$-homogeneous case the whole  $H$ is completely determined by its {\em associated ball}, i.\,e.\,the convex closed set
\[
B_H=\{z\in \R^N: H(z)\le 1\},
\]
which in general is not symmetric around the origin but always contains $0$ in its interior. Vice-versa, given any convex $K$ containing $0$ in its interior, the function 
\[
H_K(z)=\left(M_K(z)\right)^p
\]
where $M_K$ is the Minkowski functional of $K$
\[
M_K(z)=\inf \{ t>0: z/t\in K\},
\]
is a convex, positively $p$-homogeneous function with associated ball $K$. The set of symmetries of $H$ is the set of symmetries of its associated ball, namely
\[
{\rm Sym}\, (H)=\{T\in O(N): T(B_H)=B_H\}
\]
where $O(N)$ is the orthogonal group in dimension $N$. By construction, for any $T\in {\rm Sym}\, (H)$ it holds $H\circ T=H$.

\begin{proposition}\label{Lcoer}
Let $H:\R^N\to [0, \infty[$ be convex and positively $p$-homogeneous for some $p\ge 1$. Then for any $w\in \R^N$ and $\xi\in \partial H(w)$, it holds
\beq
\label{phom1}
p\, H(z)\ge (\xi, z)\qquad \forall \xi\in \partial H(z).
\eeq
\end{proposition}

\begin{proof}
Let $\xi \in \partial H(z)$ for some $z\in \R^N$. By $p$-homogeneity and convexity
\[
t^p\, H(z)=H(t\, z)\ge H(z)+(\xi, t\, z-z)
\]
for all $t\ge 1$, while the left hand side coincides with the right hand side when $t=1$. Therefore
\[
\left.\frac{d}{dt} \big[ t^p\, H(z)-H(z)-(\xi, t\, z- z)\big]  \right|_{t=1}\ge 0,
\]
which gives \eqref{phom1}.

\end{proof}

 \subsection{Energy critical points} \label{nscp}
 In this section we consider functionals of the form \eqref{J} with more general structural conditions than those required in Theorem \ref{Mth}. More precisely, we assume  that
 
  \beq
 \label{hypH}
 \begin{split}
 \text{ $H:\R^N\to [0, +\infty[$ is convex  and for some $p\ge 1$ and $0\le c< d\le \infty$ obeys}\\
 c\, |z|^p\le H(z)\le d\, |z|^p \qquad \forall z\in \R^N,\qquad \qquad\qquad\qquad \,
 \end{split}
 \eeq
  \beq
 \label{hypf0}
 \text{$f:\Omega\times \R\to \R$ is Caratheodory and}\quad
\sup_{|s|\le t}\|f(\cdot, s)\|_{L^\infty(\Omega)}<\infty \quad \forall t>0,
 \eeq
  \beq
 \label{hypf}
 \text{there exist $\theta, \mu\ge 0$, $1\le q\le p^*$ s.\,t. }  {\rm sign}\,(t)\, f(x, t)\le  \mu \, ( |t|+\theta)^{q-1} \quad \forall t\in \R. 
\eeq

Regarding $F$ as per \eqref{Ff}, it follows from \eqref{hypf0} that when $p>N$ then  $F(\cdot, v)\in L^\infty(\Omega)$ for any $v\in W^{1,p}_0(\Omega)$ by Morrey embedding, while, when $p\le N$, from \eqref{hypf} we readily obtain the one-sided estimate  
\[
\sup_\Omega F(\cdot, t)\le \frac{\mu}{q}\,\left(|t|+ \theta\right)^q
\]
for a finite $q$.
Therefore, as long as $\Omega$ has finite measure (which will be assumed henceforth), 
\beq
\label{g}
\left(H(Dv)-F(\cdot, v)\right)_-\in L^1(\Omega) \qquad \forall v\in W^{1,p}_0(\Omega).
\eeq
and the functional  $J$ in \eqref{J} is well defined on $W^{1,p}_0(\Omega)$, eventually assuming the value $\infty$ if $F(\cdot, v)\notin L^1(\Omega)$.  Fatou's lemma ensures that $J$ is l.\,s.\,c.\,on $W^{1,p}_0(\Omega)$ with the strong topology (regardless of $p\ge1$), hence the non-smooth critical point theory of \cite{DG} can be applied. Notice also that  if \eqref{hypf} is strengthened to hold for both $f$ and $-f$, 
then $J$ turns out to be locally Lipschitz, since its second term is actually $C^1$ on $W^{1,p}_0(\Omega)$, allowing to employ Clarke critical point theory \cite{Clarke}. We will not pursue these approaches here for sake of brevity  and will use the following notion of critical point, taken from \cite{DGM}.

\begin{defn}[Energy critical point]\label{defecp}
Suppose $\Omega\subseteq \R^N$ has finite measure. Under assumptions \eqref{hypH}, \eqref{hypf0} and, if $p\le N$, \eqref{hypf}, we say that $u\in W^{1,p}_0(\Omega)$ is an {\em energy critical point} for $J$, if $J(u)<\infty$ and $u$ minimises the convex functional
\beq
\label{defJu}
J_u(v)=\int_\Omega H(Dv) - f(x, u)\, v\, dx
\eeq
defined on the space
\[
V_u=\left\{v\in W^{1,p}_0(\Omega): f(\cdot, u)\, v\in L^1(\Omega)\right\}.
\]
\end{defn}

A few remarks on this notion of critical point are in order.
\begin{remark}
\label{remc}
\begin{itemize}
\item
Notice that an energy critical point $u$ is {\em assumed} to lie in $V_u$, i.\,e.\,$f(\cdot, u)\, u\in L^1(\Omega)$. The role of $p$ in the previous definition is connected to   the validity of \eqref{g} (granted, as noted, by \eqref{hypf}), so that $J$ is at least well defined. It should be remarked that, at this level of generality, it has nothing to do with \eqref{hypH}, where one could also have $c=0$ and $d=\infty$.
Regardless,  from  $H\ge 0$ and  the assumption $J(u)<\infty$,  we see that any energy critical point $u$ fulfills  $H(Du)\in L^1(\Omega)$. In particular, if $c$ in \eqref{hypH} is finite and positive and $p\le N$, then $u\in L^q(\Omega)$ by the finite measure assumption on $\Omega$ and Sobolev embedding.   Let us also remark that if $u\in W^{1,p}_0(\Omega)$ is assumed, or proved, to be bounded, assumption \eqref{hypf} is irrelevant as $J$ is always well defined as long as $\Omega$ has finite measure and in this case $V_u=W^{1,p}_0(\Omega)$. 
\item
We stress here that the notion of energy critical point is a very general one as it does not involve any topological structure on $W^{1,p}_0(\Omega)$. As such, it is useless if one wants to develop an effective critical point theory, where topological considerations are paramount.  
Under conditions  \eqref{hypH} with $c>0$ and $p>1$, as well as \eqref{hypf0} and \eqref{hypf} if $p\le N$, a critical point for   $J$ in the sense of  \cite{DG} (with the strong topology on $L^{p^*}(\Omega)$) is an energy critical point (see \cite[Theorem 5]{DGM}). Similarly, if  $d<\infty$ and $p>1$ in \eqref{hypH} and \eqref{hypf} holds for both $f$ and $-f$, then any critical point in the sense of Clarke  for $J$ is an energy critical point. 
\item
In general, \eqref{hypf0} and the finite measure assumption on $\Omega$ ensure that $f(x, u)\in L^1(\Omega)$, hence  $V_u$ contains  $C^\infty_c(\Omega)$ and is dense in $W^{1,p}_0(\Omega)$. In the case $d<\infty$ of \eqref{hypH}, being 
\beq
\label{defJH}
J_H(v):=\int_\Omega H(Dv)\, dx
\eeq
convex and finite  on $W^{1,p}_0(\Omega)$, it is locally Lipschitz, hence for any energy critical point $u$ it holds
\[
\int_\Omega f(x, u)\, (v-u)\, dx\le J_H(v)-J_H(u)\le C_r\, \|D(v-u)\|_{p}
\]
for all $v\in V_u$ with $\|D(v-u)\|_{p}\le r$. Therefore the linear functional
\[
V_u\ni v\mapsto \int_\Omega f(x, u)\, v\, dx
\]
has a unique extension $T$ to $W^{1,p}_0(\Omega)$ as an element in $W^{-1, p}(\Omega)$. Notice that while $\langle T, v\rangle$ is finite for all $v\in W^{1,p}_0(\Omega)$, this does not imply that $f(\cdot, u)\, v\in L^1(\Omega)$, unless a one-sided bound on $f(x, u)\, v$ holds true, enabling  to apply Brezis-Browder theorem. 
\end{itemize}
\end{remark}

\begin{proposition}\label{propositionEL}
Let $\Omega\subseteq\R^N$ be open with  finite measure and assume \eqref{hypH} with $d<\infty$, \eqref{hypf0} and, if $p\le N$, \eqref{hypf}.
If $u\in W^{1,p}_0(\Omega)$ is an energy critical point for $J$, there exists $\xi\in L^{p'}(\Omega; \R^N)$ such that 
\beq
\label{EL}
\begin{cases}
\text{$\xi(x)\in \partial H(Du(x))$}&\text{ for a.\,e.\,$x\in \Omega$}\\[10pt]
\displaystyle{\int_\Omega \left(\xi, D\varphi\right) - f(x, u)\, \varphi\, dx=0}&\ \forall \varphi\in V_u.
\end{cases}
\eeq

\end{proposition}

\begin{proof}

Let us compute $\partial J_H$ (see \eqref{defJH}),  following \cite{PP}. The operator 
\[
W^{1,p}_0(\Omega)\ni v\mapsto A(v)=Dv \in \left(L^p(\Omega)\right)^N
\]
is linear and continuous with transpose 
\[
\big(L^{p'}(\Omega)\big)^N\ni g\mapsto A^*(g)={\rm div}\, g\in W^{-1, p'}(\Omega)
\]
where, as usual, ${\rm div}\, g\in W^{-1, p'}(\Omega)$ acts as
\beq
\label{qwe}
\langle {\rm div}\, g, \varphi\rangle=\int_\Omega \left( g, D\varphi\right)\, dx\qquad \varphi\in W^{1, p}_0(\Omega).
\eeq
The convex functional 
\[
\left(L^p(\Omega)\right)^N\ni g\mapsto I_H(g)=\int_\Omega H(g)\, dx
\]
is locally Lipschitz thanks to $d<\infty$ in \eqref{hypH} and has  convex subdifferential (see \cite[Theorem 21]{R})
\[
\partial I_H(g)=\left\{\xi\in L^{p'}(\Omega; \R^N): \text{$\xi(x)\in \partial H(g(x))$ for a.\,e.\,$x\in \Omega$}\right\}.
 \]
Clearly $J_H=I_H\circ A$ and by \cite[Theorem 19]{R} it holds
 \[
 \partial \left(I_H\circ A\right)(u)=A^*\left(\partial I_H (A(u))\right).
 \]
 It follows that at any $v\in W^{1,p}_0(\Omega)$
 \[
 \partial J_H(v)=\left\{{\rm div}\, \xi : \xi\in L^{p'}(\Omega; \R^N) \text{ and $\xi(x)\in \partial H(Dv(x))$  a.\,e.\,in $\Omega$}\right\},
 \]
 where an element of the set on the right acts as in \eqref{qwe}.   To conclude, let $T\in W^{-1, p'}(\Omega)$ be the unique extension of 
 \[
 v\mapsto \int_\Omega f(\cdot, u)\, v\, dx, \qquad v\in V_u
 \]
 granted by Remark \ref{remc}.
 Then $u$ also minimises $v\mapsto J_H(v) - \langle T, v\rangle$ on $W^{1,p}_0(\Omega)$, by the continuity of $J_H-T$ and the density of $V_u$ in $W^{1,p}_0(\Omega)$. Therefore
 \[
 0\in \partial (J_H-T)(u)=\partial J_H(u)-T
 \]
 which implies \eqref{EL} when bracketed with $v\in V_u$.

 \end{proof}
 
 The following easy consequence will be useful later.

 \begin{corollary}
Suppose that the  assumptions of the previous Proposition  hold true and additionally that $H$ is positively $p$-homogeneous.
Then  any  energy critical point $u$ for $J$ fulfills
\beq
\label{id2}
(\xi, Du)=p\, H(Du) \qquad \text{a.\,e.\,in $\Omega$}
\eeq
 for any $\xi\in L^{p'}(\Omega; \R^N)$ obeying \eqref{EL}. In particular,
 \beq
 \label{id}
p\,  \int_\Omega H(Du)\, dx=\int_\Omega f(x, u)\, u\, dx.
\eeq
\end{corollary}

\begin{proof}
Since $u$ minimises the convex functional $J_u$ defined in \eqref{defJu}, it does so on the half-line through $u$, namely on $\{t\, u: t>0\}\subseteq V_u$. The function 
\[
h(t):=J_u(t\, u)=t^p\, \int_\Omega H(Du)\, dx - t\, \int_\Omega f(x, u)\, u\, dx
\]
is smooth and convex on $]0, \infty[$, therefore the previous statement forces $h'(1)=0$, which is \eqref{id}. Let now $\xi \in L^{p'}(\Omega; \R^N)$ fulfil \eqref{EL}, so that testing with $\varphi=u\in V_u$, we get 
\[
\int_\Omega (\xi, Du)\, dx=\int_\Omega f(x, u)\, u\, dx.
\]
By using \eqref{id}, we deduce
\[
\int_\Omega p\, H(Du)-(\xi, Du)\, dx=0,
\]
but by \eqref{phom1} the integrand is non-negative, therefore it must vanish a.\,e.\,in $\Omega$, proving \eqref{id2}.
\end{proof}

\section{Regularity}\label{Sregularity}
The following lemma enucleates the De Giorgi method for proving $L^\infty$ bounds.

\begin{lemma}\label{propositionGiusti}
Let $\Omega\subseteq \R^N$, $q\ge 1$ and $h_0> 0$. Let $u$ be such that $(u-h_0)_+\in L^q(\Omega)$ and $u$ satisfies the inequalities
\beq
\label{it}
\int_\Omega (u-k)_+^q\, dx\le \frac{b}{(k-h)^{\beta}}\left(\frac{k}{k-h}\right)^\delta\left(\int_\Omega (u-h)_+^q\, dx\right)^{1+\gamma}
\eeq
for all $h_0\le h<k$ and parameters $b, \beta, \delta\ge 0$ and $\gamma>0$. Then $u$ is bounded from above and, if  furthermore $\beta$ is positive, it holds
\beq
\label{supest}
\sup_\Omega u\le \max\left\{2\, h_0, \,  C\, \left( b^{1/\gamma} \int_\Omega (u-h_0)_+^q\, dx\right)^{\gamma/\beta}\right\}
\eeq
for $C=C(b, \gamma, \beta, \delta)>0$ smoothly depending on the data.
\end{lemma}

\begin{proof}
Let $M\ge 2\, h_0$ to be determined and set $k_n=M\, (1-2^{-1-n})$ for $n\ge 0$. Notice that $k_n\uparrow M$,  $k_0=M/2\ge h_0$ and
\[
\frac{b}{(k_{n+1}-k_n)^{ \beta}}\left(\frac{k_{n+1}}{k_{n+1}-k_n}\right)^\delta\le \frac{b}{M^{ \beta}} \, 2^{(n+2)\, ( \beta+\delta)}.
\]
We can then apply  \eqref{it} for $k=k_{n+1}$ and $h=k_{n}$ to get
\[
\int_\Omega (u-k_{n+1})_+^q\, dx\le \frac{A\, b\, A^n}{M^{\beta}}\left(\int_\Omega (u-k_n)_+^q\, dx\right)^{1+\gamma}.
\]
for $A=4^{\beta+\delta}$.
By \cite[Proposition 7.1]{Giusti}, if
\beq
\label{piccol}
\int_\Omega (u-k_{0})_+^q\, dx\le \frac{M^{\beta/\gamma}}{(A\, b)^{1/\gamma}\, A^{1/\gamma^2}}
\eeq
then 
\[
\lim_n \int_\Omega (u-k_n)_+^q\, dx=0,
\]
which forces $u\le M$ in $\Omega$. Condition \eqref{piccol} is satisfied for sufficiently large $M$, since $(u-h_0)_+\in L^q(\Omega)$ and hence, for $M\uparrow\infty$,
\[
\int_\Omega(u-k_0)_+^q\, dx=\int_\Omega(u-M/2)_+^q\, dx\to 0 .
\]
Therefore the upper boundedness of $u$ is proved. In the case $\beta>0$, observe that \eqref{piccol} is fulfilled  (and hence $u\le M$) for all $M$ such that 
\[
M\ge 2\, h_0\qquad \text{and} \qquad M\ge A^{\frac{\gamma+1}{\gamma\, \beta}}\, \left( b^{1/\gamma} \int_\Omega (u-h_0)_+^q\, dx\right)^{\gamma/\beta},
\]
proving  the bound \eqref{supest}.
\end{proof}

\begin{theorem}[$L^\infty$ bounds]\label{thbound}
Let $\Omega\subseteq \R^N$ be open with  finite measure and suppose  \eqref{hypH} holds with $c>0$, as well as \eqref{hypf0} and, if $p\le N$, \eqref{hypf}. Then, any energy critical point $u\in W^{1,p}_0(\Omega)$ of $J$
is bounded. Moreover, if \eqref{hypf} holds true (regardless of $p$) for some $q<p^*$, then
\beq
\label{supbound}
\|u\|_{L^\infty(\Omega)}\le 
 \max\left\{\theta, C\,\left((\mu/c)^{\frac{N}{p}}\int_\Omega |u|^q\, dx\right)^{\frac{p}{p\, q+ (p-q)\, N}}\right\} 
 \eeq
for some $C=C(N, p, q)$ independent of $d$, $\theta$, $\mu$ and $\Omega$. Moreover, for any fixed $N$, $C(N, \cdot, \cdot):\, [1, \infty[\times  [1, p^*[\,  \to \R_+$ is positive and continuous.
\end{theorem}

\begin{proof}
By considering $H/c$ and $f/c$, $u$ is still an energy critical point of the corresponding functional, so we can assume that $c=1$.
 Without loss of generality, it suffices to prove that $u$ is bounded above (otherwise we look at $v(x)=-u(x)$, which is a critical point of a functional with $\check{H}(z)=H(-z)$ and $\check{f}(x, t)=-f(x, -t)$ obeying the same structural conditions as $J$). The boundedness of $u$ being trivial from Morrey embedding when $p>N$, we may suppose that $p\ge 1$ is arbitrary and that \eqref{hypf} holds true. Moreover, we can assume $q<\infty$ in \eqref{hypf}, since otherwise the condition $q\le p^*$ forces  $p>N$ and again boundedness follows from Morrey embedding.  Finally, by the definition of energy critical point and the discussion following Definition \ref{defecp}, we can also assume that $H(Du)$ and $f(\cdot, u)\, u$ belong to $L^1(\Omega)$, while $u\in L^q(\Omega)$.

Given $k> 0$, we compare the minimality of $u$  with $\min\{k, u\}$, which belongs to $V_u$ since $|f(\cdot , u)\, \min\{k, u\}|\le |f(\cdot, u)\, u|\in L^1(\Omega)$. From
\[
J_u(u)\le J_u(\min\{k, u\}) 
\]
we get, by the locality of the weak gradient and \eqref{hypf},
\beq
\label{dfg}
 \int_{{\mathcal A}_k} H(Du)\, dx\le \int_{{\mathcal A}_k} f(x, u)\, (u-k)\, dx\le \mu\, \int_{{\mathcal A}_k}(u+\theta)^{q-1}\, (u-k)\, dx
\eeq
where we set
\[
{\mathcal A}_k=\{x\in \Omega: u(x)\ge k\}
\]
which has finite measure.
The inequality
\beq
\label{t2}
\mu\, (t+\theta)^{q-1}\, (t-k)\le \mu\,4^{q-1} (k^{q}+ (t-k)^q),
\eeq
 holds true for any $t\ge k\ge \theta/2$\footnote{Indeed, if $\theta/2\le k\le t\le 2\, k$ then $(t+\theta)^{q-1}\, (t-k)\le k\, (2\,k +\theta)^{q-1}\le 4^{q-1}\, k^q$, while if $t\ge 2\, k$ then $t\le 2\, (t-k)$ and thus $(t+\theta)^{q-1}\, (t-k)\le 2^{q-1}\, (t-k)^q$.}. 
Using  \eqref{hypH}, the growth condition on $f$ and \eqref{t2}, we get
\beq
\label{t4}
\int_{{\mathcal A}_k}|Du|^p\, dx\le \mu\, 4^{q-1}\, \int_{{\mathcal A}_k} k^q+(u-k)^q\, dx
\eeq
which holds for any  $k\ge \theta/2$.

We now aim at applying Lemma \ref{propositionGiusti}. First notice that for finite $q$, the inequality $q\le p^*$ is equivalent (even in the case $p\ge N$) to
\beq
\label{pjh}
q\, p+(p-q)\, N\ge 0.
\eeq
For $N\ge 2$, \eqref{pjh} implies
\[
q\le \frac{N}{N-1}\, \frac{q\, p+p-q}{p}
\]
so that we can  chain H\"older, Gagliardo-Nirenberg inequality \eqref{GN} and \eqref{t4} to get
\beq
\label{t3}
\begin{split}
\int_\Omega& (u-k)_+^q\, dx\le \left(\int_\Omega (u-k)_+^{1^*\frac{q\, p+p-q}{p}}\, dx\right)^{\frac{q\, p}{1^*\, (q\, p+p-q)}}\, |{\mathcal A}_k|^{1-\frac{q\, p}{1^*\, (q\, p+p-q)}}\\
&\le C\,  \left(\int_\Omega |D(u-k)_+|^p\, dx\right)^{\frac{q}{ q\, p+p-q}}\left(\int_\Omega (u-k)_+^q\, dx\right)^{\frac{q\, p}{p'(q\, p+p-q)}}\, |{\mathcal A}_k|^{\frac{q\, p +(p-q)\, N}{N\, (q\, p+p-q)}}\\
&\le C\, \left(\mu \,  \int_{{\mathcal A}_k} k^q+(u-k)^q\, dx\right)^{\frac{q}{ q\, p+p-q}}\left(\int_\Omega (u-k)_+^q\, dx\right)^{\frac{q\, (p-1)}{q\, p+p-q}}\, |{\mathcal A}_k|^{\frac{q\, p+(p-q)\, N}{N\, (q\, p+p-q)}}
\end{split}
\eeq
 with
 \[
 C=\left( 4^{q-1}\, G_{p, q, N}^p\right)^{\frac{q}{q\, p+p-q}}
 \]
 smoothly depending on the data.
 
Let now $\theta/2\le h<k$. By Chebishev inequality
\beq
\label{t5}
|{\mathcal A}_k|\le \frac{1}{(k-h)^q}\int_\Omega (u-h)_+^q\, dx,
\eeq
so that, since $(u-k)_+\le (u-h)_+$,
\[
\int_{{\mathcal A}_k} k^q+(u-k)^q\, dx=k^q\, |{\mathcal A}_k|+\int_\Omega (u-k)_+^q\, dx\le \left(\frac{k^q}{(k-h)^q}+1\right)\int_\Omega (u-h)_+^q\, dx,
\]
hence, since $k/(k-h)>1$,
\[
\int_{{\mathcal A}_k} k^q+(u-k)^q\, dx\le 2\left(\frac{k}{k-h}\right)^q\int_\Omega (u-h)_+^q\, dx.
\]
Inserting the latter into \eqref{t3} and using  again \eqref{t5} on the last factor gives
\[
\int_\Omega (u-k)_+^q\, dx\le \frac{C\, \mu^{\frac{q}{ q\, p+p-q}}}{(k-h)^{q\, \frac{q\, p+(p-q)\, N}{N\, (q\, p+p-q)}}}\, \left(\frac{k}{k-h}\right)^{\frac{q^2}{ q\, p+p-q}}\, \left(\int_{\Omega}(u-k)_+^q\, dx\right)^{1+\frac{q\, p}{N\, ( q\, p+p-q)}}
\]
and Lemma \ref{propositionGiusti} yields the stated boundedness with the choices
\[
b=C\, \mu^{\frac{q}{ q\, p+p-q}},\qquad  \beta=q\, \frac{q\, p+(p-q)\, N}{N\, (q\, p+p-q)}, \qquad \gamma=\frac{q\, p}{N\, (q\, p+p-q)}, \qquad \delta=\frac{q^2}{q\, p+p-q}
\]
(notice that $b, \delta, \gamma$ are positive and $\beta$ is non-negative thanks to \eqref{pjh}). Moreover if $q<p^*$ then the previous choice of $\beta$ is positive, so that \eqref{supest} provides \eqref{supbound} by noting that $(u-\theta/2)_+\le |u|$.
If $N=1$ boundedness is trivial from $u\in W^{1,p}_0(\Omega)$ and $p>1$, while the estimate \eqref{supbound} is obtained as before, this time instead of using the Gagliardo-Nirenberg inequality \eqref{GN}, employing
\[
\|v\|_{\infty}^{\frac{q\, p+p-q}{p}}\le \frac{q\, p+p-q}{p}\,  \|Dv\|_p\, \|v\|_q^{q/p'},
\]
which can be proved by elementary means. The details are omitted.
\end{proof}

 \begin{remark}[{\em On the finite measure assumption on $\Omega$}] 
 It is worth noting that the boundedness of $u$ stated in the previous Theorem does not depend quantitatively on $\Omega$, not even in the case $q=p^*$ when $p\le N$. Indeed, the same statement holds true for arbitrary $\Omega\subseteq \R^N$, not necessarily of finite measure, as long as the  energy critical point $u\in W^{1,1}_{\rm loc}(\Omega)$ is assumed to fulfill $H(Du)\in L^1(\Omega)$, $f(\cdot, u)\, u\in L^1(\Omega)$ and, additionally,  that $u$ vanishes at infinity in the measure sense. In this case, \eqref{dfg} still holds true and the proof goes on as before since each ${\mathcal A}_k$ has finite measure for $k>0$. Moreover, assuming without loss of generality $\theta>0$ in \eqref{hypf} provides the bound \eqref{supbound} with $(|u|-\theta/2)_+$ (which is in $L^q(\Omega)$ by Sobolev, since its support has finite measure) instead of $|u|$ on the right hand side.
 \end{remark}
\begin{remark}[{\em On the form of \eqref{supbound}}]
\label{remsup}
It is worth analysing the optimality of  the form \eqref{supbound} of the $L^\infty$ bound in the simple case of the Dirichlet energy, the main point being that one cannot expect the simpler bound
\beq
\label{ref2}
\|u\|_{L^\infty(\Omega)}\le  C\,\left((\mu/c)^{\frac{N}{p}}\int_\Omega |u|^q\, dx\right)^{\frac{p}{p\, q+ (p-q)\, N}}
\eeq
to be true when $\theta>0$ with a constant $C=C(N, p, q)$ independent of $\theta$.

Let $\mu \in \ ]0, \lambda_1[$, where $\lambda_1$ is the first eigenvalue of the negative Dirichlet Laplacian in the ball $B_1$. For any $\theta>0$, the problem
\beq
\label{ref}
\begin{cases}
-\Delta u= \mu\, (u+\theta) &\text{in $B_1$},\\
u=0&\text{in $\partial B_1$}
\end{cases}
\eeq
has, by standard methods, a unique positive solution, which minimises a functional obeying the assumptions of the previous Theorem with $c=1$, $q=p=2$ and $\mu, \theta$ as in \eqref{hypf}.  By comparison, $u$ is greater or equal than the solution of 
\[
\begin{cases}
-\Delta v= \mu\, \theta &\text{in $B_1$},\\
v=0&\text{in $\partial B_1$},
\end{cases}
\]
which is $v(x)=\mu\, \theta \, (1-|x|^2)/(2\, N)$, therefore
\beq
\label{ref1}
\|u\|_{L^\infty(B_1)}\ge \frac{\mu\, \theta}{2\, N}.
\eeq
On the other hand, by testing the weak form of \eqref{ref} with $u$ and by the definition of $\lambda_1$, we get
\[
\mu\, \int_{B_1} u^2+\theta\, u\, dx=\int_{B_1} |Du|^2\, dx\ge \lambda_1\, \int_{B_1} u^2\, dx
\]
so that, rearranging and using H\"older inequality, we find
\[
(\lambda_1-\mu)\int_{B_1} u^2\, dx\le \mu\, \theta \, \int_{B_1} u\, dx\le \mu\, \theta\, \left(\int_{B_1} u^2\, dx\right)^{1/2} \omega_N^{1/2}.
\]
The latter gives the bound
\[
\left(\mu^{N/2}\, \int_{B_1} u^2\, dx\right)^{1/2}\le \omega_N^{1/2}\, \frac{\mu^{1+\frac{N}{4}}\, \theta}{\lambda_1-\mu},
\]
so that, if $\mu>0$ fulfills
\[
C(N, 2, 2)\, \omega_N^{1/2}\, \frac{\mu^{\frac{N}{4}}}{\lambda_1-\mu}<\frac{1}{2\, N},
\]
by \eqref{ref1} we have
\[
\|u\|_{L^\infty(B_1)}> C(N, 2, 2)\left(\mu^{N/2}\, \int_{B_1} u^2\, dx\right)^{1/2}
\]
and  the bound \eqref{ref2} cannot hold.
\end{remark}

Once boundedness is proved, assumption \eqref{hypf0} ensures that $f(\cdot, u)\in L^\infty(\Omega)$ and therefore $V_u=W^{1,p}_0(\Omega)$. Even if \eqref{EL} is just a differential inclusion, basic regularity theory is nevertheless at hand, thanks to the theory of quasi-minima developed by Giaquinta and Giusti in \cite{GG}.

\begin{corollary}[Regularity]\label{reg}
Suppose \eqref{hypH} holds true with $0<c< d<\infty$ and $p>1$, as well as \eqref{hypf0}. Let $u$ be a bounded energy critical point for $J$ on $W^{1,p}_0(\Omega)$ and $M_0>0$ be such that  $\|u\|_{L^\infty(\Omega)}\le M_0$. If $\Omega$ is bounded and obeys 
\beq
\label{L}
L=\inf_{x\in \partial\Omega, r>0} \frac{|B_r(x)\setminus \Omega|}{|B_r(x)|}>0,
\eeq
then
\[
\|u\|_{C^\beta(\overline\Omega)}\le C
\]
 for $C>0$ and $\beta\in (0, 1)$ depending smoothly on $p$, $c$, $d$, $M_0$, $f$, $L$ and $|\Omega|$, continuously for fixed $N$. 
 If  $\Omega$ has finite measure,
 \[
 \|u\|_{C^\beta(\Omega')}\le C
 \]
 for any $\Omega'\Subset \Omega$, with $C$ additionally depending on $\Omega'$ and $\Omega$.

 \end{corollary}

\begin{proof}
To make the statement precise, we rewrite \eqref{hypf0} as 
\[
\sup_{|s|\le t}\|f(\cdot, s)\|_{L^\infty(\Omega)}\le h(t)\qquad \forall t>0
\]
for a smooth, increasing $h$, so that the constants and exponents in the statement will depend smoothly just on $M_0$, through the function $h$ which in turn depends on $f$.
As $u$ is bounded, it minimises the convex functional $J_u$ defined in \eqref{defJu} over the whole $W^{1,p}_0(\Omega)$, 
we can adapt the proof of \cite[Theorem 7.1]{Giusti} to our setting, focusing on the Dirichlet boundary condition.
We focus on the first statement, since the second one will follow from it by choosing smooth $\Omega'\Subset\Omega$ obeying \eqref{L}.

Consider the case $p\le N$.
For any ball $B_R$, not necessarily contained in $\Omega$ and $0<r<R$, we fix $\eta\in C^{\infty}_c(B_R)$ such that 
\beq
\label{defeta}
{\rm supp}\, \eta\subseteq B_R, \qquad 0\le \eta\le 1, \qquad \left.\eta\right|_{B_r}\equiv 1, \qquad |D\eta|\le \frac{C(N)}{R-r},
\eeq
 and compare the minimality of $u$ with the function $v=u-\eta\, (u-k)_+$ for $k\ge 0$. Notice that $v\in W^{1,p}_0(\Omega)$ regardless of the relative position of $B_R$ and $\Omega$, and 
\[
{\rm supp}\, (v-u)\subseteq {\mathcal A}(k, R):=\{x\in \overline{\Omega}\cap B_R: u(x)\ge k\}, 
\]
\beq
\label{der}
Dv=(1-\eta)\, Du - (u-k)_+\, D\eta\qquad \text{in\ } {\mathcal A}(k, R).
\eeq
Therefore $J_u(u)\le J_u(v)$ entails
\[
\int_{{\mathcal A}(k, R)} H(Du)\, dx\le \int_{{\mathcal A}(k, R)} H(Dv)+f(x, u)\, \eta\, (u-k)_+\, dx.
\]
By using the growth conditions on $H$, \eqref{hypf0} and the boundedness of $u$, we infer
\[
c\, \int_{{\mathcal A}(k, R)}|Du|^p\, dx\le d\, \int_{{\mathcal A}(k, R)}|Dv|^p+ M_0\, \sup_{|s|\le M_0} \|f(\cdot, s)\|_{L^\infty(\Omega)}\,  dx\\
\]
so that by \eqref{der} 
\[
\begin{split}
\int_{{\mathcal A}(k, r)}|Du|^p\, dx&\le C\, \int_{{\mathcal A}(k, R)}\left[(1-\eta)^p\,|Du|^p +|D\eta|^p\, (u-k)^p\right]\, dx +C\, |{\mathcal A}(k, R)|\\
&\le C\, \int_{{\mathcal A}(k, R)\setminus {\mathcal A}(k, r)}|Du|^p\, dx +C\, \int_{{\mathcal A}(k, R)}\frac{(u-k)^p}{(R-r)^p}\, dx +C\, |{\mathcal A}(k, R)|
\end{split}
\]
for $C=C(N, p, c, d, M_0)>0$ and continuously depending of the data. By hole-filling the first term on the right, i.\,e.\,adding to the inequality $C$ times the left hand side, we get
\[
(C+1) \int_{{\mathcal A}(k, r)}|Du|^p\, dx\le C\int_{{\mathcal A}(k, R)}|Du|^p\, dx+\frac{C}{(R-r)^p}\, \int_{{\mathcal A}(k, R)}(u-k)^p\, dx +C\, |{\mathcal A}(k, R)|, 
\]
or
\[
 \int_{{\mathcal A}(k, r)}|Du|^p\, dx\le (1-\delta)\left[\int_{{\mathcal A}(k, R)}|Du|^p\, dx +\frac{1}{(R-r)^p}\int_{{\mathcal A}(k, R)} (u-k)^p\, dx +\, |{\mathcal A}(k, R)|\right]
 \]
 for $\delta=1/(C+1)$, so that \cite[Proposition 6.1]{Giusti} provides
 \[
 \int_{{\mathcal A}(k, r)}|Du|^p\, dx\le \frac{C}{(R-r)^p}\int_{{\mathcal A}(k, R)} (u-k)^p\, dx +C\, |{\mathcal A}(k, R)|.
 \]
A similar conclusion holds true for $-u$, which is a critical point of a functional with the same structural conditions. Therefore $u$ belongs to the De Giorgi class as considered in \cite[p. 81 and p. 90]{LU}, namely  ${\mathcal B}_p(\Omega\cup \partial\Omega, M_0, C, 1, 0)$, yielding the stated H\"older regularity up to the boundary. Luckily, the H\"older exponent found in \cite[Theorem 6.1 and Proposition 7.1]{LU} is very explicit in its dependance from the parameters, showing the stated dependance of $\beta$ and on the $C^\beta(\overline\Omega)$ norm of $u$.

If $p>N$, by comparing the minimality of $u$ with $\pm u_\pm$, we directly have
 \[
 \int_{\{\pm u\ge 0\}} H(Du)\, dx\le \int_{\{\pm u\ge 0\}} f(x, u)\, u_{\pm}\, dx\le C
 \]
 with a constant $C$ smoothly depending on $M_0$.
Extending $u$ as $0$ outside $\Omega$, using the growth condition \eqref{hypH} and Morrey embedding  yields $u\in C^{1-N/p}(\R^N)$ with a uniform bound on the norm.  
\end{proof}

\begin{corollary}[Minimum principles]\label{smp}
Let \eqref{hypH} hold with $0<c< d<\infty$ and $p>1$, as well as \eqref{hypf0}, and let $u$  be a bounded energy critical point for $J$ on $W^{1,p}_0(\Omega)$, where $\Omega$ has finite measure.
\begin{enumerate}
\item
If $f(\cdot , t)\ge 0$ for $t\le 0$,
then $u\ge 0$ in $\Omega$.
\item
If $u\ge 0$ and there exist $A\ge 0$, $\delta>0$ such that for all $t\in\,  ]0, \delta[$
\beq
\label{smph}
\inf_\Omega f(\cdot, t)\ge -A\, t^{p-1},
\eeq
 then, in each connected component of $\Omega$, $u$ is either strictly positive or vanishes identically.
 \end{enumerate}
\end{corollary}

\begin{proof}
We test the minimality with respect to $J_u$ of $u$ against $u_+$, to get
\[
\int_{\{u<0\}} H(Du)\, dx\le \int_\Omega f(x, u)\, (u-u_+)\, dx=-\int_{\{u\le 0\}} f(x, u)\, u_-\, dx
\]
and since $f(\cdot, u)\ge 0$ on $u\le 0$, the term on the left must vanish, implying that $u\ge 0$ a.\,e.. To prove the strict positivity of $u$ we further test the minimality of $u$ with $v=u+\eta\, (k-u)_+$ for $k\ge 0$ and $\eta$ as in \eqref{defeta} with $B_R\subseteq \Omega$. Since 
\[
{\rm supp}\, (u-v)\subseteq \{x\in B_R: u(x)\le k\}=:{\mathcal B}(k, R)
\]
the minimality relation $J_u(u)\le J_u(v)$ reads
\[
\int_{{\mathcal B}(k, R)} H(Du)\, dx\le \int_{{\mathcal B}(k, R)} H(Dv)+ f(x, u)\, (u-v)\, dx.
\]
Since $u$ is bounded, we can assume \eqref{smph} holds true for all $t\ge 0$, with a bigger constant $A$ also depending on $\delta$ and $\|u\|_\infty$.
Since $u-v=-\eta\, (k-u)_+\le 0$, \eqref{smph} gives
\[
\begin{split}
\int_{{\mathcal B}(k, R)} H(Du)\, dx&\le \int_{{\mathcal B}(k, R)} H(Dv)+ A\, u^{p-1}\, \eta\, (k-u)_+\, dx.\\
&\le \int_{{\mathcal B}(k, R)} H(Dv)\, dx +A\, k^p\, |{\mathcal B}(k, R)|.
\end{split} 
\]
Proceeding with the same calculations of the previous proof one arrives via hole-filling at
\[
\int_{{\mathcal B}(k, r)} |Du|^p\, dx\le \frac{C}{(R-r)^p}\int_{{\mathcal B}(k, R)}(k-u)_+^p\, dx +A\, k^p\, |{\mathcal B}(k, R)| 
\]
for all $k\ge 0$. Therefore $u\ge 0$ belongs to the homogeneous De Giorgi class $DG^-_p(\Omega)$ as in \cite{DBT}. A weak Harnack inequality holds true in this setting (see \cite[Theorem 2]{DBT}), from which the strong minimum principle readily follows.
\end{proof}

\section{Dirichlet eigenvalues}
\label{DE}

For $H:\R^N\to [0, \infty[$ convex and obeying \eqref{hypH} we can define its {\em first Dirichlet eigenvalue} as
\beq
\label{lambda1}
\lambda_{1, H}(\Omega)=\inf\left\{\int_\Omega H(Dv)\, dx: v\in W^{1,p}_0(\Omega), \int_\Omega |v|^p\, dx=1\right\}.
\eeq
A {\em first normalized Dirichlet eigenfunction} for $\lambda_{1, H}(\Omega)$ is a function solving the minimum problem for $\lambda_{1, H}(\Omega)$ with its constraints. When $\Omega$ has finite measure, the embedding $W^{1,p}_0(\Omega)\hookrightarrow L^p(\Omega)$ is compact, therefore, if  in \eqref{hypH} we assume $c>0$ and $p>1$, problem \eqref{lambda1} is always solvable and its value is positive.

The value \eqref{lambda1} makes sense in general, but we focus here on the $p$-positively homogeneous setting for $p>1$. In this case, $\lambda_{1, H}(\Omega)$ is related to the best constant for the inequality
\beq
\label{bc}
\int_\Omega |v|^p\, dx\le C_{\rm opt}(H, \Omega)\, \int_\Omega H(Dv)\, dx\qquad \forall v\in W^{1, p}_0(\Omega)
\eeq
which, by positive $p$-homogeneity, satisfies $C_{\rm opt}(H, \Omega)=1/\lambda_{1, H}(\Omega)$. Moreover, if $u$ solves problem \eqref{lambda1}, any {\em positive} multiple of $u$ attains the best constant in \eqref{bc}, and will be called a {\em first Dirichlet eigenfunction}.
Notice that we are not assuming any regularity on $H$ except the natural local Lipschitz one and, more substantially, we do not assume that $H$ is even.
In the possibly non-even setting, it is convenient to define
\beq
\label{lambdapm}
\begin{split}
\lambda_{1, H}^+(\Omega)&=\inf\left\{\int_\Omega H(Dv)\, dx: v\in W^{1,p}_0(\Omega), \int_\Omega |v|^p\, dx=1, \ v\ge 0\right\}\\
\lambda_{1, H}^-(\Omega)&=\inf\left\{\int_\Omega H(Dv)\, dx: v\in W^{1,p}_0(\Omega), \int_\Omega |v|^p\, dx=1, \ v\le 0\right\},
\end{split}
\eeq
which may in principle be different. We say that $u\ge 0$ is a {\em first normalized Dirichlet positive eigenfunction} for $\lambda^{+}_{1, H}(\Omega)$ if $u$ realises the minimum for $\lambda_{1, H}^+(\Omega)$ (with its constraints) and that $u\le 0$ is a {\em first normailzed negative  Dirichlet eigenfunction} if it does so for $\lambda_{1, H}^-(\Omega)$. Similar considerations as in the unconstrained eigenvalue $\lambda_{1, H}(\Omega)$ are valid for sign-constrained ones and we will drop the term normalized accordingly.  

 Notice that, without any evennes assumption on $H$, it may {\em a-priori} no longer be true that if $u$ is a first positive eigenfunction then $-u$ is a first negative eigenfunction. Moreover, it holds
\beq
\label{checkH}
\lambda_{1, H}^-(\Omega)=\lambda_{1, \check{H}}^+(\Omega), \qquad \check{H}(z)=H(-z),
\eeq
since $v$ obeys the constraints involving $\lambda_{1, H}^-(\Omega)$ if and only if $-v$ does so for $\lambda_{1, H}^+(\Omega)$.

\begin{proposition}
Suppose $H:\R^N\to [0, \infty[$ is continuous and positively $p$-homogeneoues, for $p\ge 1$. Then, for any open $\Omega\subseteq \R^N$  
\beq
\label{minlambda}
\lambda_{1, H}(\Omega)=\min\left\{\lambda_{1, H}^+(\Omega), \lambda_{1, H}^-(\Omega)\right\}.
\eeq
\end{proposition}

\begin{proof}
 Clearly 
 \[
 \lambda_{1, H}(\Omega)\le \min\left\{\lambda_{1, H}^+(\Omega), \lambda_{1, H}^-(\Omega)\right\},
 \]
 so we prove the opposite inequality. To this end, notice that all the eigenvalues are finite and, for arbitrary $\eps>0$, choose $v$ such that $\|v\|_p=1$ and
 \[
\int_\Omega H(Dv)\, dx< \lambda_{1, H}(\Omega)+\eps.
 \]
If $v_-\equiv 0$ a.\,e., then $v\ge 0$ and  
\[
\min\left\{\lambda_{1, H}^+(\Omega), \lambda_{1, H}^-(\Omega)\right\}\le \lambda_{1, H}^+(\Omega)\le \int_\Omega H(Dv)\, dx<\lambda_{1, H}(\Omega)+\eps,
\]
so that letting $\eps\downarrow 0$ we have
\beq
\label{ref}
\min\left\{\lambda_{1, H}^+(\Omega), \lambda_{1, H}^-(\Omega)\right\}\le \lambda_{1, H}(\Omega).
\eeq
A similar conclusion holds  if $v_+=0$, so we can assume that both $v_+$ and $v_-$ do not vanish identically. Thus $v_+/\|v_+\|_p$ and $-v_-/\|v_-\|_p$ are admissible competitors for the corresponding problems \eqref{lambdapm}, giving, by $p$-homgeneity
\[
\lambda_{1, H}^\pm(\Omega)\, \|v_\pm\|_p^p \le  \int_\Omega H( D(\pm v_\pm))\, dx.
\]
From $1=\|v\|_p=\|v_+\|_p+\|v_-\|_p$ and $H(0)=0$ we infer 
\[
\begin{split}
\min\left\{\lambda_{1, H}^+(\Omega), \lambda_{1, H}^-(\Omega)\right\}&=\min\left\{\lambda_{1, H}^+(\Omega), \lambda_{1, H}^-(\Omega)\right\} \left(\|v_+\|_p^p+\|v_-\|_p^p\right)\\
&\le \int_\Omega H(Dv_+)\, dx+\int_\Omega H(D(-v_-))\, dx\\
&\le \int_\Omega H(Dv)\, dx\le \lambda_{1, H}(\Omega)+\eps
\end{split}
\]
so that \eqref{ref} follows again by letting $\eps\downarrow 0$.

\end{proof}

In general, we do not expect that the two values in \eqref{lambdapm} coincide, although we have not been able to produce an explicit example. If equality is always true, however,  it cannot follow from convexity, $p$-homogeneity and modularity of the energy alone, as the following toy example shows.

\begin{ex}
Let $H:\R^2\to [0, \infty[$ be
\[
H(x, y)= a\, x_+^2 + x_-^2+a\, y_+^2+ y_-^2
\]
for some $a>1$.
We look at $\R^3$ as the set of functions $u:\Omega\to \R$ with $\Omega=\{0, 1, 2\}$, whose "gradient" is $Du:=(u(0)-u(1), u(1)-u(2))\in \R^2$. The corresponding positive and negative Dirichlet eigenvalue problems are, respectively,
\[
\lambda_{1, H}^\pm(\Omega)=\inf\left\{\sum H(Du) :u(0)=u(2)=0, u(1)=\pm1\right\}.
\]
The function $H$ is convex, positively $2$-homogeneous and modular, meaning
\[
H(Du)+H(Dv)=H(D\max\{u, v\})+H(D\min\{u, v\}),
\]
 due to \cite[Remark 2.2]{GM}, so the truncation arguments employed in the previous proof still stand. However
\[
\lambda_{1, H}^+(\Omega)=1<a= \lambda_{1, H}^-(\Omega).
\]
A similar construction can clearly be generalised to finite sets $\Omega$ containing an arbitrary number of points.
\end{ex}

A simple case where the sign-constrained eigenvalues are equal is when $\Omega$ enjoys a symmetry opposite with $H$.
In order to make a precise statement, let
\[
{\rm Sym}\, (\Omega)=\left\{T\in O(N):  T(\Omega)=\Omega+x_0 \text{ for some $x_0\in \R^N$} \right\}.
\]
\begin{proposition}
Under the assumptions of the previous proposition, suppose that 
\beq
\label{sym}
{\rm Sym}\, (\Omega)\cap -{\rm Sym}\, (H)\ne \emptyset.
\eeq
Then 
\beq
\label{propositionlambda}
\lambda_{1, H}^+(\Omega)=\lambda_{1, H}^-(\Omega)=\lambda_{1, H}(\Omega).
\eeq
\end{proposition}

\begin{proof}
Fix  $T\in {\rm Sym}\, (\Omega)\cap -{\rm Sym}\, (H)$, so that it holds
\[
H(-T(z))=H(z)\qquad \forall z\in \R^N.
\]
Then, for any $v\ge 0$ in $W^{1,p}_0(\Omega)$ with $\|v\|_p=1$, the function $w=-v\circ T$ is non positive, belongs to $W^{1,p}_0(\Omega-x_0)$ for some $x_0\in \R^N$ and fulfills $\|w\|_p=1$. By translation invariance 
\[
\lambda_{1, H}(\Omega)=\lambda_{1, H}(\Omega-x_0)
\]
and
\[ 
\int_{\Omega-x_0} H(Dw)\, dx=\int_{\Omega-x_0} H(-T Dv\circ T)\, dx=\int_\Omega H(Dv)\, dx.
\]
Therefore $\lambda_{1, H}^+(\Omega)\ge \lambda_{1, H}^-(\Omega)$. Proceeding symmetrically from $v\le 0$ we get the opposite inequality, while applying \eqref{minlambda} concludes the proof of \eqref{propositionlambda}.
\end{proof}

\begin{remark}
By the translation invariance of the energy, the same conclusion holds true if each connected component of $\Omega$ separately fulfills assumption \eqref{sym}.
In particular, \eqref{propositionlambda} holds true if each connected component has at least one axis of symmetry.  Using this fact it is readily shown that  \eqref{propositionlambda} always holds true for arbitrary  open  $\Omega\subseteq \R$. Indeed, open connected  subset of $\R$ are intervals, half-lines or the whole $\R$. But if $\Omega$ has an unbounded connected component, a scaling argument shows that all eigenvalues vanish, while if $\Omega$ has only bounded connected components the previous Proposition applies, since open intervals are symmetric with respect to their midpoint. 
\end{remark} 

\begin{proposition}\label{regeigen}
Suppose that $H:\R^N\to [0, +\infty[$ is convex, positively $p$-homogeneous for some $p>1$, vanishing only at the origin and that $\Omega$ is connected with finite measure. If $u$ is a first Dirichlet eigenfunction (sign-constrained or not), then $u\in C^\alpha(\Omega)$ and never vanishes in $\Omega$. 
\end{proposition}

%

\begin{proof}
We first observe that any eigenfunction is a critical point for a functional of the form \eqref{J} obeying \eqref{hypH}, \eqref{hypf0} and \eqref{hypf}. Since $H$ vanishes only at the origin and $H$ is continuous
\[
0<\inf_{|z|=1} H(z)\le \sup_{|z|=1} H(z)<\infty,
\]
so that by the positive $p$-homogeneity of $H$, \eqref{hypH} holds with $0<c\le d<\infty$.
In the unconstrained case, for the functional 
\[
J(v)=\int_\Omega H(Dv)-\lambda_{1, H}(\Omega) |v|^p\, dx
\]
it  holds $J\ge 0$ on $W^{1, p}_0(\Omega)$ by construction and $J(u)=0$, hence $u$ minimises $J$ and is therefore an energy critical point for it.  In the sign-constrained case, by considering $\check{H}$ instead of $H$ and using \eqref{checkH}, it suffices to consider the case when $u$ is a positive eigenfunction.
The functional 
\[
J_+(v)=\int_\Omega H(Dv)-\lambda_{1, H}^+(\Omega)\,  v_+^p\, dx
\]
vanishes at $u$ and it holds $J_+(v)\ge J_+(v_+)$. But $J_+(v_+)\ge 0$ by the definition of $\lambda^+_{1, H}(\Omega)$, so again $u$ minimises $J_+$. Clearly in all cases \eqref{hypf} holds true with $q=p$  and Theorem \ref{thbound} implies boundedness, and thus continuity by Corollary \ref{reg}, of all eigenfunctions. Moreover,  Corollary \ref{smp} applies for positive eigenfunction, since in this case $f(x, t)=t_+^{p-1}\ge 0$, yielding the first statement for positive eigenfunctions.

 For an arbitrary eigenfunction $u$, we claim that either $u_+$ or $-u_-$ are eigenfunctions for $\lambda^{\pm}_{1, H}(\Omega)$ respectively, regardless of any assumption on $\Omega$. Indeed, it suffices to prove it for normalized eigenfunctions, so let $u$ be a first Dirichlet eigenfunction for $\lambda_{1, H}(\Omega)$ such that $\|u\|_p=1$. If $u$ is of constant sign, the claim is trivial. Otherwise,  both $u_+$ and $u_-$ are non trivial and then $\pm u_\pm/\|u_{\pm}\|_p$ are admissible for the corresponding problems in \eqref{lambdapm}. By the positive $p$-homogeneity of $H$
\[
\lambda_{1, H}^\pm(\Omega)\le \frac{1}{\|u_\pm\|_p^p}\, \int_\Omega H(D(\pm u_\pm))\, dx
\]
which, summed up,  imply 
\[
\begin{split}
\lambda_{1, H}(\Omega)&=\min\left\{\lambda_{1, H}(\Omega)^+, \lambda_{1, H}^-(\Omega)\right\}\left(\|u_+\|_p^p+\|u_-\|_p^p\right)\\
&\le \int_\Omega H(Du_+)+ H(D(-u_-))\, dx=\lambda_{1, H}(\Omega),
\end{split}
\]
so that all inequalities above are actually equalities, proving the claim. It follows that if $u_+\ne 0$, then $u_+$ is a positive eigenfunction and therefore never vanishes by the previous part of the proof, so $u_-= 0$ and $u=u_+>0$ in $\Omega$. A similar conclusion holds true if $u_-\ne 0$, yielding $u=-u_-<0$ in $\Omega$ and concluding the proof.

\end{proof}

\section{Main result}

In order to prove a Brezis-Oswald comparison principle, we'll need the following generalisation to the non-smooth setting of the Picone inequality.

\begin{lemma}[Picone inequality]\label{thpicone}
Suppose $H$ is convex and positively $p$-homogeneous. Let $u, v\in  W^{1,1}_{\rm loc}(\Omega)\cap L^\infty_{\rm loc}(\Omega)$ be such that $\inf_\Omega u>0$, $\inf_\Omega v\ge 0$ and let  $\xi\in {\mathcal L}(\Omega; \R^N)$ be such that $\xi(x)\in \partial H(Du(x))$ a.\,e.. Then it holds
\beq
\label{picone}
\frac{1}{p}\left(\xi , D \frac{v^p}{u^{p-1}}\right)\le H(Dv)\qquad \text{a.\,e.\,in $\Omega$. }
\eeq
Suppose additionally that  $H$ is strictly convex, $\Omega$ is connected and equality holds a.\,e.\,in \eqref{picone}. Then $v=k\, u$ for some $k>0$.
\end{lemma}

\begin{proof}
By considering $v+\eps$ instead of $v$ and then passing to the limit for $\eps\downarrow 0$, we can suppose that $\inf_\Omega v>0$.
First observe that by the assumptions $v^p/u^{p-1}\in W^{1,1}_{\rm loc}(\Omega)$ and  Liebnitz rule for weak derivatives ensures
\beq
\label{crule}
D\frac{v^p}{u^{p-1}}=(1-p)\, \frac{v^p}{u^p}\, Du + p\, \frac{v^{p-1}}{u^{p-1}}\, Dv.
\eeq
Let $\xi\in \partial H(w)$ for some $w\in \R^N$. The following computation are henceforth made at a point $x\in \Omega$ such that  $\xi=\xi(x)\in \partial H(Du(x))$. In the inequality
\[
H(z)\ge H(Du)+(\xi, z-Du)\qquad \forall z\in \R^N
\]
we choose $z=(u/v)\, Dv$ to get, by $p$-homogeneity,
\[
\left(\frac{u}{v}\right)^p\, H(Dv)=H\left(\frac{u}{v} \, Dv\right)\ge H(Du)+\frac{u}{v}\, (\xi, Dv)-(\xi, Du)
\]
and by \eqref{phom1} at $w=Du(x)$
\[
\left(\frac{u}{v}\right)^p\, H(Dv)\ge \left(\frac{1}{p}-1\right)(\xi, Du)+\frac{u}{v}\, (\xi, Dv).
\]
Multiplying both sides by $(v/u)^p$ and using \eqref{crule} we finally get
\[
H(Dv)\ge \frac{1}{p} \left(\xi, (1-p)\, \frac{v^p}{u^p}\, Du + p\, \frac{v^{p-1}}{u^{p-1}}\, Dv\right)=\frac{1}{p}\left(\xi, D \frac{v^p}{u^{p-1}}\right)
\]
and \eqref{picone} is proved. 

To prove the second statement, suppose equality holds in \eqref{picone} a.\,e.\,in $\Omega$ for some choice of $\xi$ obeying $\xi\in \partial H(Du)$ a.\,e.. Then all the previous inequalities are equalities and in particular
\[
H\left(\frac{u}{v} \, Dv\right)= H(Du)+ \left(\xi, \frac{u}{v}\, Dv- Du\right)
\]
 holds for a.\,e.\,$x\in\Omega$. Let, for any such $x$, $w=Du(x)$ and $z=(u(x)/v(x))\, Dv(x)$ and suppose that $w\ne z$. From $H(z)=H(w)+(\xi, z-w)$ and the convexity of $H$ we infer
\[
H\big(t\, z+(1-t)\, w\big)\le t\, H(z)+(1-t)\, H(w)=H(w)+\big(\xi, t\, (z-w)\big),
\]
while from $\xi\in \partial H(z)$ we deduce
\[
H(w)+\big(\xi, t\, (z-w)\big)=H(w)+\big(\xi, (t\, z+(1-t)\, w) -w\big)\le H\big(t\, z+(1-t)\, w\big).
\]
 Therefore $H$ is linear on the segment from $z$ to $w$, contradicting the strict convexity of $H$. Therefore $z=w$, hence
\[
D\log u=D\log v
\]
a.\,e.\,in $\Omega$, which implies that $u$ and $v$ are proportional by the connectedness of $\Omega$.
\end{proof}

\begin{remark}
By inspecting the proof, it is readily seen that assumption $\inf_\Omega u>0$ can be weakened to 
\[
\inf_{\Omega'} u>0, \qquad \forall \Omega'\Subset \Omega.
\]
\end{remark}

We can now prove our main result, Theorem \ref{Mth}, which we now restate in a form convenient for its proof.
\begin{theorem}\label{thfin}
Let $\Omega\subseteq \R^N$ be open, connected and of finite measure.
For $p>1$ suppose that 
\begin{itemize}
\item[i)] $H:\R^N\to [0, \infty[$ is convex, positively $p$-homogeneous and vanishes only at the origin 
\item[ii)] $f$ fulfills \eqref{hypf0} 
\item[iii)]
 The map $ t\mapsto f(x, t)/t^{p-1}$ is non-increasing on $\R_+$ for a.\,e.\,$x\in \Omega$.
\end{itemize}
Let $u\in W^{1, p}_0(\Omega)\setminus \{0\}$ be a non-negative energy critical point for $J$. The following assertions hold.
\begin{enumerate}
\item
 $u$ minimises $J$ among non-negative functions in  $W^{1,p}_{0}(\Omega)$.  
   \item
 If  $t\mapsto f(x, t)/t^{p-1}$ is strictly decreasing, $u$ is the unique positive energy critical point for $J$.
 \item
If $H$ is strictly convex, any other non-negative energy critical point is a constant multiple of $u$.
 \item
 If  $H$ is strictly convex and $f$ doesn't depend on $x$, either $u$ is the unique positive energy critical point for $J$ or $u$ is a first Dirichlet positive eigenfunction. 

 \end{enumerate}
\end{theorem}

\begin{proof}
Since $u\ge 0$ by assumption, we can assume $f(x, t)\equiv f(x, 0)$ for all $t\le 0$.
Let 
\[
h(t)=\sup_{0\le s\le t}\|f(\cdot, s)\|_\infty.
\]
 By the monotonicity assumption {\em iii)} and \eqref{hypf0}
\beq
\label{t22}
f(x, t)\le
\begin{cases} 
h(1)\, t^{p-1}& \forall t\ge 1\\
h(1)&\forall t\le 1.
\end{cases}
\eeq
so that Theorem \ref{thbound} and  Corollary \ref{reg}  apply and $u\in C^0(\Omega)\cap L^\infty(\Omega)$. Moreover
\beq
\label{t20}
f(x, t)\ge\frac{ f(x,\|u\|_\infty)}{\|u\|_\infty^{p-1}}\, t^{p-1}\ge -\frac{h(\|u\|_\infty)}{\|u\|_\infty^{p-1}}\, t^{p-1} \qquad \forall t\in [0, \|u\|_\infty]
\eeq
so that the connectedness of $\Omega$ and Corollary \ref{smp} ensure that $u>0$ in $\Omega$.

Since $u$ is bounded, it holds $V_u=W^{1,p}_0(\Omega)$, hence its Euler-Lagrange equation \eqref{EL} holds true for any $\varphi\in W^{1,p}_0(\Omega)$. Let $\xi\in L^{p'}(\Omega; \R^N)$ be such that 
\[
\xi(x)\in \partial H(Du(x))\quad \text{a.\,e.},\qquad \int_{\Omega}\left(\xi, D\varphi\right) -f(x,u)\, \varphi\, dx=0\quad \forall \varphi\in W^{1,p}_0(\Omega)
\]
and let $v\in W^{1,p}_0(\Omega)\cap L^\infty(\Omega)$ be non-negative. It is readily checked that for any $\eps>0$ the function $\varphi=v^p/(u+\eps)^{p-1}$ is an admissible test function in $W^{1,p}_0(\Omega)$, hence
\[
\int_\Omega \left(\xi, D \frac{v^p}{(u+\eps)^{p-1}}\right)\, dx=\int_\Omega f(x, u)\, \frac{v^p}{(u+\eps)^{p-1}}\, dx
\]
and by Picone inequality \eqref{picone}  (notice that $\xi\in \partial H(Du)=\partial H(D(u+\eps))$), we have
\beq
\label{t21}
p\, \int_\Omega H(Dv)\, dx\ge \int_\Omega \left(\xi, D \frac{v^p}{(u+\eps)^{p-1}}\right)\, dx=\int_\Omega f(x, u)\, \frac{v^p}{(u+\eps)^{p-1}}\, dx.
\eeq
We claim that 
\beq
\label{t16}
\lim_{\eps\downarrow 0} \int_\Omega \frac{f(x, u)}{(u+\eps)^{p-1}}\, v^p\, dx=\int_\Omega \frac{f(x, u)}{u^{p-1}}\, v^p\, dx.
\eeq
Indeed, let
\[
\Omega_+=\big\{x\in \Omega: f(x, u(x))\ge 0\big\}, \qquad \Omega_-=\big\{x\in \Omega: f(x, u(x))<0\big\}.
\]
By Beppo-Levi monotone convergence theorem, \eqref{t16} holds true with $\Omega_+$ instead of $\Omega$, while  on $\Omega_-$ it holds, by \eqref{t20},
\[
0\ge \frac{f(x, u)}{(u+\eps)^{p-1}}\, v^p\ge -\frac{h(\|u\|_\infty)}{\|u\|_\infty^{p-1}}\, \frac{u^{p-1}}{(u+\eps)^{p-1}}\, v^p\ge  -\frac{h(\|u\|_\infty)}{\|u\|_\infty^{p-1}}\, \|v\|_\infty^p
\]
so that dominated convergence yields \eqref{t16} on $\Omega_-$ as well.
Therefore \eqref{t16} is proved, and passing to the limit as $\eps\downarrow 0$ in \eqref{t21} yields
\beq
\label{t7}
p\, \int_\Omega H(Dv)\, dx\ge \int_\Omega \frac{f(x, u)}{u^{p-1}}\, v^p\, dx.
 \eeq
 Another monotone convergence argument (separately on $\Omega_+$ and $\Omega_-$) allows to remove the assumption $v\in L^\infty(\Omega)$, so that \eqref{t7} holds for all $v\in W^{1,p}_0(\Omega)$, $v\ge 0$. As a byproduct of this argument, let us notice that the integrand on the right of \eqref{t7} is always in $L^1(\Omega)$, for any $v\in W^{1, p}_0(\Omega)$, $v\ge 0$.

 Let us prove assertion {\em (1)}. To this end, recall that $F(\cdot, v)_+\in L^1(\Omega)$ for all non-negative  $v\in W^{1,p}_0(\Omega)$, thanks to the one-sided growth condition \eqref{t22}. Therefore it suffices to compare $J(u)$ with $J(v)$ for non-negative  $v\in W^{1,p}_0(\Omega)$ such that $F(\cdot, v)\in L^1(\Omega)$.
We rewrite \eqref{t7}  as
 \[
 J(v)\ge \int_\Omega \frac{f(x, u)}{u^{p-1}}\, \frac{v^p}{p}-F(x, v)\, dx
 \]
 and  recall that,  by \eqref{id},
 \[
 J(u)=\int_\Omega\frac{1}{p}\, f(x, u)\, u-F(x, u)\, dx,
 \]
 so we have
 \[
 J(v)-J(u)\ge \int_\Omega  \frac{f(x, u)}{u^{p-1}}\, \left(\frac{v^p}{p}-\frac{u^p}{p}\right)+ F(x, u)-F(x, v)\, dx.
 \]
 We claim that the integrand is non-negative. Indeed, for a.\,e.\,$x\in \Omega$ it holds
 \[
 F(x, t)-F(x, s)=\int_s^t \frac{f(x, \sigma)}{\sigma^{p-1}}\, \sigma^{p-1}\, d\sigma 
 \]
 so, considering separately the two cases $t\ge s\ge 0$ and $0\le t\le s$ and using the monotonicity of $\sigma\mapsto f(x, \sigma)/\sigma^{p-1}$, we get
 \[
 F(x, t)-F(x, s)\ge \frac{f(t)}{t^{p-1}}\, \left(\frac{t^p}{p}-\frac{s^p}{p}\right)
 \]
 for all $t, s\ge 0$. Setting $t=u(x)$ and $s=v(x)$ proves the claim, therefore $J(v)\ge J(u)$ for all $v\ge 0$, $v\in W^{1,p}_0(\Omega)$, concluding the proof of the first statement.

To prove  statement {\em (2)}, suppose now that $v$ is another nontrivial, non-negative critical point  for $J$, so that \eqref{t7} holds true with $u$ and $v$ exchanged. Recalling \eqref{id} we get
\[
\int_\Omega f(x, v)\, v\, dx\ge \int_\Omega \frac{f(x, u)}{u^{p-1}}\, v^p\, dx, \qquad \int_\Omega f(x, u)\, u\, dx\ge  \int_\Omega \frac{f(x, v)}{v^{p-1}}\, u^p\, dx
\]
which, summed up, give
\[
\int_\Omega \left(\frac{f(x, u)}{u^{p-1}}-\frac{f(x, v)}{v^{p-1}}\right)\, (u^p-v^p)\, dx\ge 0.
\]
By assumption {\em ii)}, the latter forces 
\beq
\label{t24}
\frac{f(x, u)}{u^{p-1}}=\frac{f(x, v)}{v^{p-1}}
\eeq
a.\,e.\,in $\Omega$, which implies that $u=v$ if $t\mapsto f(x, t)/t^{p-1}$ is strictly decreasing.

To prove assertion {\em (3)}, recall that by Picone inequality
\[
p\, H(Dv)-\left(\xi, D \frac{v^p}{(u+\eps)^{p-1}}\right)\ge 0,
\]
for all $\eps\ge 0$. Hence, chaining  Fatou lemma, the weak form of the equation for $u$, \eqref{t16}, \eqref{t24} and \eqref{id}, we get
  \[
  \begin{split}
\int_\Omega p\, H(Dv)-\left(\xi, D \frac{v^p}{u^{p-1}}\right)\, dx&\le \lim_{\eps\downarrow 0}\int_\Omega p\, H(Dv)-\left(\xi, D \frac{v^p}{(u+\eps)^{p-1}}\right)\, dx\\
&= \lim_{\eps\downarrow 0}\int_\Omega p\, H(Dv)-f(x, u)\, \frac{v^p}{(u+\eps)^{p-1}}\, dx\\
&= \int_\Omega p\, H(Dv)-\frac{f(x, u)}{u^{p-1}}\, v^p\, dx\\
&= \int_\Omega p\, H(Dv)-f(x, v)\, v\, dx=0
\end{split}
\]
so that equality must occur a.\,e.\,in the Picone inequality. If $H$ is strictly convex, Lemma \ref{thpicone} ensures that 
\beq\label{eq:prop}
u=k\, v \,,\qquad\mbox{for some $k>0$,}
\eeq
giving assertion {\em (3)}.

Finally, statement {\em (4)} amounts to prove  that, in the case $f(x, t)=f(t)$, the validity of \eqref{eq:prop} for some  $k\ne 1$ forces $u$ to be a first Dirichlet positive eigenfunction.
Let $g(t):=f(t)/t^{p-1}$ and assume that $k>1$ in \eqref{eq:prop}. For any fixed $\delta\in \ ]0, \|u\|_\infty[$, by continuity $u(\Omega)$ is connected and contains $]0, \|u\|_\infty-\delta]$. Hence for any $n\in \N$ there exists $x_n\in \Omega$ such that 
\[
u(x_{n})=\frac{\| u \|_\infty-\delta}{k^n}
\]
and thus \eqref{t24} and \eqref{eq:prop} give
\[
\begin{split}
g(\|u\|_\infty-\delta)&=g(u(x_0))=g(v(x_0))=g(u(x_0)/k)=g((\|u\|_\infty-\delta)/k)=g(u(x_1))\\
&=g(v(x_1))=g(u(x_1)/k)=g((\|u\|_\infty-\delta)/k^2)=\dots=g((\|u\|_\infty-\delta)/k^n)
\end{split}
\]
for any $n\ge 0$. Therefore $g$, being continuous and non-increasing, is constantly equal to some $\lambda\ge 0$ on $]0, \|u\|_\infty-\delta]$. By letting $\delta\downarrow 0$, the same conclusion holds on $]0, \|u\|_\infty]$. Similarly, if $0<k<1$, we infer that $g\equiv \lambda\ge 0$ on 
\[
]0, \|v\|_\infty]=\ ]0, \|u\|_\infty/k]\supseteq \ ]0, \|u\|_\infty].
\] 
Notice that then $f(t)=\lambda\, t^{p-1}$ for $t\in \left[0, \|u\|_\infty\right]$, hence 
\beq
\label{F}
F(t)=\frac{\lambda}{p}\, t^p\qquad \text{ for all  $t\in [0, \|u\|_\infty ]$}.
\eeq
 By \eqref{id} it holds 
 \[
 J(u)=\int_\Omega H(Du)-\frac{\lambda}{p} \, u^p\, dx=0
 \]
  and by the first assertion of the Theorem
 \[
  J(v)\ge J(u)\qquad \forall v\in W^{1,p}_0(\Omega), \ v\ge 0.
 \]
 If $0\le v\le \|u\|_\infty$, from the last two displays and \eqref{F} we infer
 \[
 \frac{\lambda}{p} \int_\Omega v^p\, dx\le \int_\Omega H(Dv)\, dx.
 \]
By positive $p$-homogeneity, the latter  holds for any bounded $v\ge 0$, $v\ne 0$, (simply writing it for $w=v\, \|u\|_\infty/\|v\|_\infty$, which fulfills $0\le w\le \|u\|_\infty$) and thus for any $v\ge 0$ in $W^{1,p}_0(\Omega)$ by a monotone convergence argument. It follows that $\lambda/p\le \lambda_{1, H}^+(\Omega)$ and $J(u)=0$ implies
 \[
\lambda_{1, H}^+(\Omega)\, \int_\Omega u^p\, dx\le \int_\Omega H(Du)\, dx=\frac{\lambda}{p}\, \int_\Omega u^p\, dx \le \lambda_{1, H}^+(\Omega)\, \int_\Omega u^p\, dx
 \]
 i.\,e., $u$ is a first Dirichlet positive eigenfunction, as claimed.
 
 \end{proof}

\bigskip

\bigskip

{\bf Acknowledgements}\\
The author thanks an anonymous referee for his careful reading of the manuscript and valuable suggestions.
Prof. M. Degiovanni is warmly thanked for outlining a mistake in a first version of the manuscript and for drawing our attention to paper \cite{DGM2}. Prof. J.\,I.\,D\'iaz is also thanked for kindly pointing out reference \cite{D}.

\medskip

{\bf Funding informations}\\
The author is member of {\em Gruppo Nazionale per l'Analisi Ma\-te\-ma\-ti\-ca, la Probabilit\`a e le loro Applicazioni} (GNAMPA),
 is partially supported by project PIACERI - Linea 2 and 3 of the University of Catania and by GNAMPA's project "Equazioni alle derivate parziali di tipo ellittico o parabolico con termini singolari". 

\medskip

{\bf Conflict of interest}\\
The author states no conflict of interest

\end{document}